\documentclass[11pt]{article}

\usepackage{amsmath}
\usepackage{amssymb}
\usepackage{amsthm}
\usepackage{color}
\usepackage{comment}
\usepackage{graphics}
\usepackage[english]{babel}
\usepackage[T1]{fontenc}
\usepackage{lmodern}
\usepackage{enumerate}

\newtheorem{thm}{Theorem}[section]
\newtheorem{prop}{Proposition}[section]
\newtheorem{coro}{Corollary}[section]
\newtheorem{lemma}{Lemma}[section]
\newtheorem{rem}{Remark}[section]

\numberwithin{equation}{section}

\newcommand{\R}{\mathbb{R}}             
\newcommand{\N}{\mathbb{N}}             
\newcommand{\C}{\mathbb{C}}             
\renewcommand{\H}{\mathcal{H}}          

\newcommand{\e}{\epsilon}

\newcommand{\half}{\frac{1}{2}}

\newcommand{\ds}{\displaystyle}

\title{Exponential localization of Steklov eigenfunctions on warped product manifolds: the flea on the elephant phenomenon}

\author{Thierry Daud\'e, Bernard Helffer and Fran\c{c}ois Nicoleau}

\begin{document}

\maketitle

\begin{abstract}
  This paper is devoted to the analysis of Steklov eigenvalues and Steklov eigenfunctions on a class of warped product Riemannian manifolds $(M,g)$ whose boundary $\partial M$ consists in two distinct connected components $\Gamma_0$ and $\Gamma_1$. First, we show that the Steklov eigenvalues can be divided into two families $(\lambda_m^\pm)_{m \geq 0}$ which satisfy accurate asymptotics as $m \to \infty$. Second, we consider the associated Steklov eigenfunctions which are the harmonic extensions of the boundary Dirichlet to Neumann eigenfunctions. In the case of symmetric warped product, we prove that the Steklov eigenfunctions are exponentially localized on the whole boundary $\partial M$ as $m \to \infty$. Whenever we add an asymmetric perturbation to a symmetric warped product, we observe a flea on the elephant effect. Roughly speaking, we prove that "half" the Steklov eigenfunctions are exponentially localized on one connected component of the boundary, say $\Gamma_0$, and the other half on the other connected component $\Gamma_1$ as $m \to \infty$. 
\end{abstract}



\newpage 

\section{Introduction and main results}

\subsection{History of the problem}

Let $(M,g)$ be a smooth compact connected Riemannian manifold of dimension $n \geq 2$ with boundary $\partial M$. 
The Steklov spectrum corresponds to the eigenvalues of the Dirichlet to Neumann (\textbf{DN}) operator $\Lambda_g(\omega)$ defined by 
\begin{equation} \label{DN}
	\Lambda_g(\omega) \psi = (\partial_\nu u)_{|\partial M}, 
\end{equation}	
where for all $\psi \in H^{\half}(\partial M)$, $u \in H^1(M)$ is the unique solution of the Dirichlet problem
\begin{equation} \label{DirichletPb}
  \left\{ \begin{array}{cl}
  	-\Delta_g u = \omega \,u, & \textrm{on} \ M, \\
  	u = \psi, & \textrm{on} \ \partial M.
  \end{array} \right.	
\end{equation}	
Here, $\omega$ is a fixed frequency that does not belong to the Dirichlet spectrum of $-\Delta_g$, the positive Laplace-Beltrami operator on $(M,g)$. 
It is well known (see for instance \cite{Sa2013, Uh2014}) that the DN operator $\Lambda_g(\omega)$ is an elliptic, first-order, selfadjoint operator 
on $L^2(\partial M)$ and thus, that its spectrum is discrete and forms an increasing sequence 
$$
\lambda_0 < \lambda_1 \leq \dots \leq \lambda_m \to +\infty\,.
$$  
Denote by $(\phi_m)_{m \geq 0}$ the corresponding normalized eigenfunctions in $L^2(\partial M)$ and by $(\varphi_m)_{m \geq 0}$ their so-called {$\omega$-}harmonic extensions, 
\textit{i.e} the solutions of (\ref{DirichletPb}) with $\psi = \phi_m$. The functions $(\varphi_m)_{m \geq 0}$ are called the Steklov eigenfunctions and will be the main object of interest of this paper. 

In the case of a bounded domain $ M \subset \R^n$ and zero frequency $\omega = 0$, Hislop and Lutzer proved in \cite{HiLu2001} that Steklov eigenfunctions concentrate at the boundary $\partial M$ as $m \to +\infty$. Precisely, they proved

\begin{thm}[Hislop, Lutzer (2001)]
Let $(\phi_m)_{m \geq 0}$ be the sequence of normalized eigenfunctions of $\Lambda_g(0)$ satisfying $\Lambda_g(0) \,\phi_m = \lambda_m \,\phi_m$ and $\| \phi_m \|_{L^2(\partial M)} = 1$ for all $m\geq0$, and $(\varphi_m)_{m\geq0}$ be their harmonic extensions to $M$, \textit{i.e.} the solutions of (\ref{DirichletPb}) with $\psi=\phi_m$. Then  for any compact set $K \subset \overset{\circ}{M}$, then, as $m\to \infty$, 
$$
  \| \varphi_m \|_{H^1(K)} = \mathcal O (m^{-\infty})\,. 
$$
\end{thm}

\vspace{0.2cm}
They conjectured moreover that "the decay should be actually of order $O(e^{-\textrm{dist}(K,\partial M) |m|})$ in the case of real-analytic metric $g$ up to the boundary $\partial M$". This conjecture has been studied recently in the two dimensional case by Polterovich, Sher and Toth in \cite{PST2019}. These authors proved that there exist some positive constants $\tau$ and $C$ depending only on the geometry of a Riemannian surface $M$ such that
$$
  | \varphi_m (x)  | \leq C e^{-\tau \,\textrm{dist}(x,\partial M) \,\lambda_m}\mbox{ as } m \to  \infty,
$$
where $\textrm{dist}(x,\partial M)$ denotes the Riemannian distance from $x$ to $\partial M$.
In the case of higher dimensional real analytic Riemannian manifolds up to the boundary, a similar (though local) result was obtained by Galkowski and Toth in \cite{GaTo2019}. Precisely, they proved:

\begin{thm}[Galkowski-Toth (2019)] \label{GT}
Assume that $(M,g)$ is a real-analytic compact Riemannian manifold with real-analytic boundary $\partial M$. Then for any  $\delta > 0$, there exists $0 < \e = \e(M,g,\delta)$ such that, if $(\phi_m)_{m\geq0}$ is a sequence of normalized eigenfunction of $\Lambda_g(0)$ satisfying $\Lambda_g(0) \,\phi_m = \lambda_m \phi_m$ and $\| \phi_m \|_{L^2(\partial M)} = 1$, then their harmonic extensions $\varphi_m$ 
satisfy the exponential decay estimate
$$
  | \partial_x^\alpha  \varphi_m(x) | \leq C_{\alpha,\delta} \ \lambda_m^{\frac{n}{2} - \frac{3}{4} + |\alpha|}  e^{-d(x) \lambda_m} \mbox{ for } { \rm dist}(x,\partial M) < \e,
$$
where $C_{\alpha,\delta} > 0$ is a constant independent of $m$ and
$$
  d(x) = {\rm dist} (x,\partial M) + (C_{M,g} - \delta)\ {\rm dist}^2(x,\partial M)\,.
$$
Here 
$$
  C_{M,g} = -\frac{3}{2} + \half \inf_{(x',\xi') \in S^* \partial M} Q(x',\xi'),
$$ 
where $Q$ is the symbol of the second fundamental form of $\partial M$.     
\end{thm}

Note that the above accurate pointwise exponential localization of the Steklov eigenfunctions $\varphi_m$ as $m \to \infty$ only holds in a collar neighbourhood of the boundary $\partial M$. However, using additionally the maximum principle for the Laplace equation, one can prove that there exist positive  constants $\tau,C, m_0$ such that for all $m \geq m_0$
$$
  | \varphi_m(x) | \leq C e^{-\tau \, \lambda_m}, \quad \textrm{dist}(x,\partial M) \geq \e.  
$$ 
Two remarks in the papers \cite{GaTo2019, PST2019} motivated the present paper. 
\begin{enumerate}
\item In \cite{PST2019}, through the study of the example of an annular 2D domain, Polterovich, Sher and Toth noticed that the Steklov eigenfunctions aren't localized on the \emph{whole} boundary in general, but can be localized on certain connected components of the boundary only. They call \emph{dominant} and \emph{residual} boundary components the parts of the boundary where the Steklov eigenfunctions concentrate or decay exponentially as $m \to \infty$ respectively.  	

\item From a heuristic point of view, Galkowski and Toth interpret the above exponential localization results as a { tunnelling} effect of the Steklov eigenfunctions from the boundary $\partial M$ (where they concentrate microlocally on the cosphere bundle $S^* \partial M = \{(x,\xi) \in T^*\partial M, |\xi|_g = 1\}$) into the interior of the manifold $M$ (where the Steklov eigenfunctions satisfy $ -\Delta_g \varphi_m = \omega \varphi_m$ and thus must concentrate at the zero section $\{ (x,\xi) \in T^* M, \ |\xi|_g = 0 \}$ as $m \to \infty$). We refer to the introduction of \cite{GaTo2019}, p.~2 for this interpretation. 
\end{enumerate}

In this paper, we would like to develop and precise these two observations through the analysis of the Steklov eigenfunctions on warped product Riemannian manifolds whose boundaries have two  distinct connected components. Precisely, we consider a cylinder $M = [0,1] \times K$ with $K$ a $(n-1)$-dimensional smooth closed manifold equipped with a Riemannian metric having the form
\begin{equation} \label{Metric}
	g = f(x) [dx^2 + g_K],
\end{equation}
where $f$ is a smooth positive function on $[0,1]$ and $g_K$ is any smooth Riemannian metric\footnote{Sometimes, it will be enough to assume that $f \in C^k, \ k \geq 2$ in our main results. Similarly, we do not  need to assume that $g_K$ is smooth on $K$, but only that $g_K$ is uniformly elliptic on $K$ with $C^2$ coefficients.} on $K$. Such a Riemannian manifold is called a 
\emph{warped product} and  $f$ is called  the corresponding \emph{warping function}. Note that the boundary $\partial M$ of $M$ is disconnected since it consists in two copies of $K$:
$$
\partial M = \Gamma_0 \cup \Gamma_1, \quad \Gamma_0 = \{0\} \times K, \quad \Gamma_1 = \{1\} \times K.
$$

For such simple models that contain in particular the cases of disks, cylinders and annulus, we shall be able to improve the results of \cite{HiLu2001, PST2019, GaTo2019} in several directions: \\

1. The warped products are not supposed to be real-analytic up to the boundary, but $C^\infty$, and even less regular. More precisely, some of our results only require a certain regularity $C^k([0,1]), \ k \geq 2$ on the horizontal metric coefficients and only $C^2$ on the transversal metric coefficients. \\

2.  We slightly improve the pointwise exponential localization estimates of the Steklov eigenfunctions as $m \to \infty$ in the case of warped products. In particular, we connect the notion of dominant and residual boundary components of \cite{PST2019} to a slight asymmetry (with respect to $\half$) of the warping function $f$ under consideration.  \\ 

3. Better than that, we put into evidence Simon's "flea on the elephant phenomenon" for the localization of the Steklov eigenfunctions on the boundary (see \cite{BGZ2019, HeSj1985, JLMS1981, Si1985}) and appendix \ref{FleaElephant}). Precisely, we prove that for symmetric (with respect to $\half$) warped products, the Steklov eigenfunctions concentrate at both connected components of the boundary as $m \to \infty$. But we also show that any  \emph{generic asymmetric} perturbation breaks this picture ! Indeed, we prove that for asymmetric warped products, the Steklov eigenfunctions either concentrate on the connected component $\Gamma_0$ of the boundary $\partial M$, or concentrate on the other connected component $\Gamma_1$ as $m \to \infty$.

\subsection{Main results}

In order to state our main results, let us introduce a few additional notations. Under our hypotheses on the transversal Riemannian manifold $(K,g_K)$, the associated Laplace-Beltrami operator $-\Delta_K$ is a second-order elliptic selfadjoint operator on $K$ which has a discrete spectrum $(\mu_m)_{m\geq0}$ ordered (counting multiplicity) by : 
$$
0 = \mu_0 < \mu_1 \leq \mu_2 \leq \dots \leq \mu_m \longrightarrow +\infty\,. 
$$
We denote by $(Y_m)_{m \geq 0}$ a sequence of normalized eigenfunctions of $-\Delta_K$ associated with  the eigenvalues $(\mu_m)_{m \geq 0}$, \textit{i.e}
$$
-\Delta_K \,Y_m = \mu_m Y_m\,. 
$$

The warped product structure of $(M,g)$ entails that the transversal Laplace-Beltrami operator $-\Delta_K$ commutes with the Laplace-Beltrami operator $-\Delta_g$ on the whole 
manifold $(M,g)$, and also with the self-adjoint DN operator $\Lambda_g(\omega)$ on $L^2(\partial M)$, (in what follows, we shall identify $L^2(\partial M)$ with $L^2 (K) \times L^2(K)$).
As a consequence, we shall see below that we can associate to each eigenvalue $\mu_m, \ m\geq0$, 
of the transversal Laplacian \emph{two} distinct Steklov eigenvalues $\lambda_m^\pm$,  (with $\lambda_m^+ > \lambda_m^-$), 
the set of pairs $(\mu_m,\lambda_m^\pm)$ forming the joint spectrum of $(-\Delta_K,\Lambda_g(\omega))$. \\

{
 
 We shall also show that the Steklov eigenfunctions $\varphi_m^\pm$ associated with  $\lambda_m^\pm$ have the product structure
\begin{equation} \label{ProductStructure}
	\varphi_m^\pm(x,\theta) =  f^{\frac{2-n}{4}}(x) w_m^\pm(x) Y_m(\theta), \quad \forall m \geq 0, \quad \forall (x,\theta) \in [0,1] \times K, 
\end{equation}
where the $w_m^\pm(x)$ satisfy the 1D Schr\"odinger equation  
\begin{equation} \label{RadialODE}
   \left\{ \begin{array}{l} 
   -(w_m^\pm)'' + q_f(x) w_m^\pm = -\mu_m w_m^\pm, \quad x \in [0,1], \\
    w_m^\pm(0) = \phi^{\pm,0}_m, \quad w_m^\pm(1) = \phi^{\pm,1}_m,
   \end{array} \right. 
\end{equation}
with
$$
  q_f(x) = \frac{\left(f^{\frac{n-2}{4}}\right)''(x)}{f^{\frac{n-2}{4}}(x)} - \omega f(x),
$$
and where $\phi_m^\pm = (\phi^{\pm,0}_m Y_m, \phi^{\pm,1}_m Y_m)$ are the normalized eigenfunctions 
of $\Lambda_g(\omega)$ associated with  the eigenvalues $(\lambda_m^\pm)_{m\geq0}$\,.

\begin{rem}[On the multiplicity of Steklov eigenvalues]
Let us emphasize that the multiplicity of the Steklov eigenvalues $\lambda_m^\pm$ comes from the multiplicity of the eigenvalues $\mu_m$ of the Laplace-Beltrami operator on the transversal manifold $(K,g_K)$. In the case of non simple Steklov eigenvalues, we thus do not have uniqueness in the choice of the corresponding Steklov eigenfunctions. However, among the eigenspaces $E_m^\pm$ associated with  an eigenvalue $\mu_m$ of multiplicity $\ell_m >1$, all Steklov eigenfunctions can still be written as a product of functions depending on $x$ and on $\theta$ respectively. Indeed, note that the ODE (\ref{RadialODE}) only depends on $\mu_m$ and not on the choice of the transversal eigenfunctions $Y_m$. Thus any {eigenfunction} $\Phi_m^\pm$ belonging to $E_m^\pm$ has the form
$$
  \Phi_m^\pm (x,\theta) = f^{\frac{2-n}{4}}(x) w_m^\pm(x) \Theta_m(\theta), \quad \forall (x,\theta) \in [0,1] \times K, 	
$$ 
where $\Theta_m$ is a finite linear combination of the $Y_j$'s belonging to the eigenspace associated with  the eigenvalue $\mu_m$.  

Since our main results only depend on the product structure (\ref{ProductStructure}) 
of the Steklov eigenfunctions, we shall identify in what follows the Steklov eigenfunctions with those given by a particular choice of the orthonormal basis of eigenfunctions $(Y_m)_{m\geq 0}$.  
\end{rem}

Our main results concern the asymptotic behavior of the Steklov eigenvalues $\lambda_m^\pm$ and of their gap $d_m:=|\lambda_m^+-\lambda_m^-|$, 
as well as the localization of the Steklov eigenfunctions $\varphi_m^\pm$. We shall distinguish the case $n=2$ and $\omega=0$ for which we have explicit formulas for the above quantities, 
and the cases $n=2, \ \omega \ne 0$, or $n \geq 3$. 

Let us start with the case $n=2$ and $\omega = 0$, where we recall that  $q_f = 0$. We prove :

\begin{thm}[$n=2, \ \omega=0$] \label{Main20} ~\\
1. If $f(0) = f(1)$, then as $m \to \infty$
\begin{eqnarray*}
  \lambda_m^\pm &=& \frac{\sqrt{\mu_m}}{f(0)} + \mathcal{O}(\sqrt{\mu_m} e^{-\sqrt{\mu_m}}),  \\
  d_m &=& \frac{2}{\sqrt{f(0)}} \frac{\sqrt{\mu_m}}{\sinh(\sqrt{\mu_m})}. 
\end{eqnarray*}	
Moreover, there exist two constants $C>0$ and $m_0 > 0$ such that for all $m \geq m_0$ and for all $(x,\theta) \in [0,1] \times K$
$$ 
|\varphi_m^\pm(x,\theta)| \leq C \left( e^{-\sqrt{\mu_m} x} + e^{-\sqrt{\mu_m} (1-x)} \right).
$$ 
2. If $f(0) \ne f(1)$, then as $m \to +\infty$
\begin{eqnarray*}
  \lambda_m^+ &=& \max \left(\frac{1}{\sqrt{f(0)}}, \frac{1}{\sqrt{f(1)}}\right) \sqrt{\mu_m} + \mathcal{O}(\sqrt{\mu_m} e^{-\sqrt{\mu_m}}), \\ 
  \lambda_m^- &=& \min\left(\frac{1}{\sqrt{f(0)}}, \frac{1}{\sqrt{f(1)}}\right) \sqrt{\mu_m} + \mathcal{O}(\sqrt{\mu_m} e^{-\sqrt{\mu_m}}), \\
  d_m &=& \left| \frac{1}{\sqrt{f(0)}} - \frac{1}{\sqrt{f(1)}} \right| \sqrt{\mu_m} + \mathcal{O}(\sqrt{\mu_m} e^{-2\sqrt{\mu_m}}).
\end{eqnarray*}	
Moreover, if $f(0) < f(1)$, there exist two constants $C>0$ and $m_0 > 0$ such that for all $m \geq m_0$ and for all $(x,\theta) \in [0,1] \times K$
\begin{align*}	
&| \varphi_m^+(x,\theta) | \leq C \left( e^{-\sqrt{\mu_m}x} + e^{-\sqrt{\mu_m}(2-x)} \right), \\
& | \varphi_m^-(x,\theta) | \leq C \left( e^{-\sqrt{\mu_m}(1+x)} + e^{-\sqrt{\mu_m}(1-x)} \right),
\end{align*}
and the same equalities hold with $+$ and $-$ inverted if $f(0) > f(1)$. 
\end{thm}
In addition we see that if $f(0)=f(1)$, then the Steklov eigenfunctions are exponentially localized at both boundaries $\Gamma_0$ and $\Gamma_1$ as $m \to \infty$, whereas if $f(0) \ne  f(1)$, then half the Steklov eigenfunctions (for instance the ones indexed by $+$) are exponentially localized at $\Gamma_0$ and the other half (indexed by $-$) are exponentially localized at $\Gamma_1$.  

\vspace{0.3cm}
We turn now to the main results concerning the other cases $n=2, \ \omega \ne 0$ or $n \geq 3$ for which the potential $q_f$ does not vanish identically. We divide our results 
into four distinct cases. \\
}
		
\noindent \textbf{Case I [($\mathbf{n=2, \ \omega \ne 0}$) or ($\mathbf{n\geq3}$): symmetric warped product ]}. We first assume that the warping function is symmetric, that is $f(x) = f(1-x)$ for all $x \in [0,\half]$. In this case, we prove

\begin{thm}[Case I] \label{MainSymmetric}
Under the assumption of symmetry on $f$, we have {as $m\to \infty$}
$$
  \lambda_m^\pm = \frac{\sqrt{\mu_m}}{\sqrt{f(0)}} + \mathcal O (1) \,,
$$
and
$$
\lambda_m^+ -\lambda_m^- =    \frac{2}{\sqrt{f(0)}} \sqrt{\mu_m} e^{-\sqrt{\mu_m}}  \left(1 + \mathcal O \left(\frac{1}{\sqrt{\mu_m}}\right)\right)\,.
$$
Moreover, there exist two constants $C>0$ and $m_0 > 0$ such that for all $m \geq m_0$ and for all $(x,\theta) \in [0,1] \times K$
$$ 
	|\varphi_m^\pm(x,\theta)| \leq C \mu_m^\frac{n-2}{4} \left( e^{-\sqrt{\mu_m} x} + e^{-\sqrt{\mu_m} (1-x)} \right).
$$ 
\end{thm}
We thus see that in the case of a \emph{symmetric} warped product, the Steklov eigenfunctions are exponentially localized and equidistributed at \emph{both} boundaries $\Gamma_0$ and $\Gamma_1$ as $m \to \infty$. Note that we only need to assume $f \in C^2([0,1])$ in that situation.\\

\vspace{0.3cm}
\noindent \textbf{Case II [($\mathbf{n=2, \ \omega \ne 0}$) or ($\mathbf{n\geq3}$): asymmetric warped product ]}. The situation differs drastically whenever we add an asymmetric warping function to the previous situation. Precisely, let us assume now that $f(x) = f_0(x) + f_1(x)$ where $f_0$ is symmetric with respect to $\half$ and $f_1$ is an asymmetric perturbation of $f_0$. We shall  distinguish three different subcases: \\

\textbf{Case II.A}. We assume here that the Taylor series of $f(x)$ and $f(1-x)$ differ at the order $k$ at $x=0$, \textit{i.e.} $k \geq 0$ is the smallest integer such that $f^{(k)}(0)  \ne (-1)^k f^{(k)}(1)$.  \\

\noindent
First, let us define
\begin{eqnarray}
	a_0 &=&b_0= \frac{1}{\sqrt{f(0)}} - \frac{1}{\sqrt{f(1)}}, \nonumber \\
	a_k &=& -\frac{\omega}{2^{k+1}\sqrt{f(0)} } \ \left( f^{(k)}(0) -(-1)^k f^{(k)}(1)   \right)  \mbox{ if }  k\geq 1, \nonumber \\
	b_k &=& \frac{n-2}{2^{k+1} f(0)^{3/2}} \ \left( f^{(k)}(0) -(-1)^k f^{(k)}(1)   \right)  \mbox{ if }  k\geq 1.
\end{eqnarray}

\noindent
We shall see that the localization of the Steklov eigenfunctions depends heavily on the sign of the nonzero constants $a_k$ for $n=2$ and $\omega \not=0$, or $b_k$ for $n\geq3$.  \\ 
More precisely, we have the following result 

\begin{thm}[Case II.A] \label{MainIIA}
Under the above assumptions, we have as $m\to \infty$,
$$
\lambda_m^+ = \max \left( \frac{1}{\sqrt{f(1)}},\frac{1}{ {\sqrt{f(0)}}}\right) \sqrt{\mu_m} + \mathcal O (1) , \quad 
\lambda_m^- = \min \left( \frac{1}{\sqrt{f(1)}},\frac{1}{ {\sqrt{f(0)}}}\right) \sqrt{\mu_m} + \mathcal O (1).
$$
The distance between the two eigenvalues $\lambda_m^\pm$ satisfies in dimension $n=2$ with $\omega \not=0$,
\begin{eqnarray}\label{dmsym1}
	d_m &=&  |a_0| \ \sqrt{\mu_m} +  \mathcal O \left(  1 \right)   \mbox{ if } \ k=0, \\
	d_m &=& |a_k| \ \mu_m^{-\frac{1+k}{2}} +  \mathcal O  \left(  \mu_m^{-\frac{k+2}{2}}\right) \mbox{ if } \ k \geq 1  ,
\end{eqnarray}
whereas in dimension $n \geq 3$,
\begin{equation}\label{dmsym3}
	d_m =  |b_k| \ \mu_m^{\frac{1-k}{2}} +  \mathcal O \ \left(  \mu_m^{-\frac{k}{2}}\right)   \mbox{ if  }  k \geq 0.
\end{equation}
Moreover, there exist constants $C_k>0$ and $m_0 > 0$ such that for all $m \geq m_0$ and for all $(x,\theta) \in [0,1] \times K$, 
we have 
\begin{enumerate}
\item For $n=2$ with $\omega \not =0$, if $a_k <0$, 
	\begin{align*}
		& |\varphi_m^+(x,\theta)| \leq C_0 \left( e^{-\sqrt{\mu_m}(1+x)} + e^{-\sqrt{\mu_m}(1-x)} \right),  \mbox{ if } \ k =0\, .\\
		& |\varphi_m^-(x,\theta)| \leq C_0 \left(  e^{-\sqrt{\mu_m}x} +  e^{-\sqrt{\mu_m}(2-x)} \right), \mbox{ if } \ k =0 \,.\\
		& |\varphi_m^+(x,\theta)| \leq C_k \left( \mu_m^{\frac{k + 2}{2}} e^{-\sqrt{\mu_m}(1+x)} + e^{-\sqrt{\mu_m}(1-x)} \right),  \mbox{ if } \ k \geq 1\,. \\
		& |\varphi_m^-(x,\theta)| \leq C_k \left(   e^{-\sqrt{\mu_m}x} + \mu_m^{\frac{k + 2}{2}} e^{-\sqrt{\mu_m}(2-x)} \right), \mbox{ if } \ k \geq 1\,.
	\end{align*}		
	and, the same estimates hold with $+$ and $-$ inverted if $a_k>0$\,. 	
\item For  $n \geq 3$, if $b_k <0$, 
\begin{align*}
		& |\varphi_m^+(x,\theta)| \leq C_k \left( \mu_m^{\frac{n+2k - 2}{4}} e^{-\sqrt{\mu_m}(1+x)} + \mu_m^{\frac{n - 2}{4}} e^{-\sqrt{\mu_m}(1-x)} \right), \\
		& |\varphi_m^-(x,\theta)| \leq C_k \left(  \mu_m^{\frac{n - 2}{4}} e^{-\sqrt{\mu_m}x} + \mu_m^{\frac{n +2k - 2}{4}} e^{-\sqrt{\mu_m}(2-x)} \right),
\end{align*}		
and, the same estimates hold with $+$ and $-$ inverted if $b_k>0$\,. 	

\end{enumerate}
\end{thm}

Hence, if the Taylor series of $f(x)$ and $f(1-x)$ differ at the order $k$ at $x=0$, then for large enough $m$, "half" the Steklov eigenfunctions (for instance the ones indexed by $+$) are exponentially localized at $\Gamma_1$, and the other "half" (the ones indexed by $-$) are exponentially localized at $\Gamma_0$. Note that we only need to assume $f \in C^\ell([0,1])$ with $\ell =\max(2,k)$ in that situation. \\

\noindent
Finally, as a by-product, we deduce

\begin{coro}\label{orderinfty}
If $f(x)$ and $f(1-x)$ have the same Taylor series at $x=0$\,,  one has as $m \to \infty$,
\begin{equation*}
\lambda_m^+ -\lambda_m^- =  \mathcal O (\mu_m^{-\infty}).
\end{equation*}
\end{coro}

\textbf{Case II.B}. We assume now that there exists $a \in [0,\half[$ such that $f(x) = f(1-x)$ for all $x \in [0,a]$. If $a=0$, then we assume only that $f(x)$ and $f(1-x)$ have the same Taylor series at $x=0$. Moreover, we assume that there exists $0<\delta <\half - a$ such that for all $x \in ]a,a+\delta]$, we have either $q_f(x) - q_f(1-x) > 0$, or $q_f(x) - q_f(1-x) < 0$. We can prove in this case :

\begin{thm}[Case II.B] \label{MainIIB}
	Under the above assumptions, if there exists $\delta>0$ such that $q_f(x) - q_f(1-x) > 0$ for all $x \in ]a,a+\delta]$, we have as $m \to  \infty$,
	$$
	\lambda_m^- = \frac{\sqrt{\mu_m}}{\sqrt{f(0)}} + \mathcal O (1) , \quad \lambda_m^+ = \frac{\sqrt{\mu_m}}{\sqrt{f(1)}} + \mathcal O (1) . 
	$$
	If $a>0$, for all $\e>0$ small enough, there exists three constants $c_\e>0$, $C_\e >0$, and $m_\e > 0$ such that for all $m \geq m_\e$.
	$$
	c_\e \ e^{-2(a+\e) \sqrt{\mu_m}}\leq \lambda_m^+ -\lambda_m^- \leq C_\e \  e^{-2 (a-\e) \sqrt{\mu_m} }, 
	$$
	If $a=0$, for all $\e>0$ small enough, there exists two constants $c_\e>0$ and $m_{\e} > 0$ such that for all $m \geq m_{\e}$,
	$$
	\lambda_m^+ -\lambda_m^- \geq c_\e \ e^{-2\e \sqrt{\mu_m}}.
	$$

	\noindent
	Moreover, for all $m \geq m_\e$ and for all $(x,\theta) \in [0,1] \times K$, 
	\begin{align*}
		& |\varphi_m^-(x,\theta)| \leq C_\e \left( e^{-\sqrt{\mu_m}(1 - 2(a+\e)+x)} + \sqrt{\mu_m}^{\frac{n-2}{2}} e^{-\sqrt{\mu_m}(1-x)} \right), \\
		& |\varphi_m^+(x,\theta)| \leq C_\e \left(  \sqrt{\mu_m}^{\frac{n-2}{2}} e^{-\sqrt{\mu_m}x} +  e^{-\sqrt{\mu_m}(2 - 2(a+\e)-x)} \right).
	\end{align*}
	If for all $x \in ]a,a+\delta]$, we have $q_f(x) - q_f(1-x) < 0$, the same estimates hold with $+$ and $-$ inverted. 	
\end{thm}

Hence, we once again prove that half the Steklov eigenfunctions (for instance the ones indexed by $+$) are exponentially localized at $\Gamma_0$, and the other half (indexed by $-$) are exponentially localized at $\Gamma_1$ as $m \to \infty$. We also show the precise dependence of these estimates on the interval $[0,a]$ on which $f(x)$ and $f(1-x)$ coincide. This shows that the decay rate of the Steklov eigenfunctions at the residual boundaries crucially depends on the support of the asymmetric perturbation $f_1(x)$. Observe at last that $f$ is supposed to be $C^\infty$ on $[0,1]$ and that the sign conditions $q_f(x) - q_f(1-x) > 0$, or $q_f(x) - q_f(1-x) < 0$, can be interpreted as mere conditions on the scalar curvature $S_g$ of $(M,g)$  (see (\ref{Hyp-q1})-(\ref{Hyp-q2}) and Remark \ref{cepsilon}). \\

\vspace{0.3cm}
\textbf{Case II.C}. We only assume that $f(x)$ and $f(1-x)$ have the same Taylor series at $x=0$. {Our results are less complete but we are able to prove:}

\begin{thm}[Case II.C] \label{MainIIC}
Under the above assumptions, for all $m \geq 0$,
$$
\lambda_m^+ -\lambda_m^- \geq \frac{2 \ \mu_m}{(f(0))^{1/2}} \ \frac {1}{\sqrt{\mu_m} \sinh (\sqrt{\mu_m}) +  e^{\sqrt{\mu_m} + ||q_f||} }  ,
$$
where $||q_f||$ is the $L^2$-norm of the potential $q_f$.
 \\ 
Moreover, there exists a subsequence $(m_k)_{k\geq 0}$  such that we have {either}
$$
\lambda_{m_k}^- = \frac{\sqrt{\mu_{m_k}}}{\sqrt{f(0)}} + \mathcal O (1)\, , \quad \lambda_{m_k}^+ = \frac{\sqrt{\mu_{m_k}}}{\sqrt{f(1)}} + \mathcal O (1) \mbox{ as } m_k \to \infty. 
$$
and moreover, there exists a constant $C > 0$ such that for all $\e > 0$, there exists $m_\e > 0$ such that for all $m_k \geq m_\e$ and for all $(x,\theta) \in [0,1] \times K$
\begin{align*}
		& |\varphi_{m_k}^-(x,\theta)| \leq C \sqrt{\mu_{m_k}}^{\frac{n-2}{2}} \left( \e e^{-\sqrt{\mu_{m_k}}x} + e^{-\sqrt{\mu_{m_k}}(1-x)} \right), \\
		& |\varphi_{m_k}^+(x,\theta)| \leq C \sqrt{\mu_{m_k}}^{\frac{n-2}{2}} \left(  e^{-\sqrt{\mu_{m_k}}x} + \e e^{-\sqrt{\mu_{m_k}}(1 - x}  \right),
\end{align*}	 	
{or the same estimates hold with $+$ and $-$ inverted.	}
\end{thm}	
Hence, we can only prove in this last case that there exists a sequence of Steklov eigenfunctions $\varphi_{m_k}^\pm$ such that the eigenfunctions with superscript $+$  are localized at $\Gamma_0$ , and the eigenfunctions with superscript  $-$  are  localized at $\Gamma_1$.  Note moreover that the warping function $f$ is assumed to be in $C^\infty([0,1])$ in this case.

\begin{rem}
In Theorems \ref{Main20}, \ref{MainSymmetric}, \ref{MainIIA}, \ref{MainIIB} and \ref{MainIIC}, we use the variable $x$ to state our localization results since it is the natural variable in our description of $(M,g)$. However, it would be more intrinsic to use $\mathrm{dist}(p,\partial M)$ the Riemannian distance from $p \in M$ to the boundary $\partial M$. Note that for $p=(x,\theta) \in M$, we have
$$
  \mathrm{dist}(p,\partial M) = \min(\mathrm{dist}(p,\Gamma_0), \mathrm{dist}(p,\Gamma_1)), 
$$  
where
$$  
  \mathrm{dist}(p,\Gamma_0) = d_0(x) =\int_0^x \sqrt{f(s)} ds, \quad \mathrm{dist}(p,\Gamma_1) = d_1(x) = \int_x^1 \sqrt{f(s)} ds. 
$$
Using that $f$ is positive and smooth on $[0,1]$, we obtain the following approximations of $x$ in terms of $d_0(x)$ and $d_1(x)$ at the order $2$
$$
  x = \frac{1}{\sqrt{f(0)}} \left( d_0(x) + \half \kappa_0 d_0^2(x) + \mathcal O(d_0^3(x)) \right) \mbox{ as }  x \to 0, 
$$
$$
  1-x = \frac{1}{\sqrt{f(1)}} \left( d_1(x) + \half \kappa_1 d_1^2(x) + \mathcal O(d_1^3(x)) \right) \mbox{ as } x \to 1,
$$
where $$ \kappa_0 = - \frac{f'(0)}{4f^{3/2}(0)}\,, \ \kappa_1 =  \frac{f'(1)}{4f^{3/2}(1)}$$ are the second fundamental forms of $\Gamma_0$ and $\Gamma_1$ respectively (see Petersen \cite{Pet2016} chapter 3, warped product). We thus observe that the leading terms $e^{-\sqrt{\mu_m} x}$ and $e^{-\sqrt{\mu_m}(1-x)}$ obtained in the estimates of Theorems \ref{MainSymmetric}, \ref{MainIIA}, \ref{MainIIB} and \ref{MainIIC} become
$$
  e^{-\frac{\sqrt{\mu_m}}{\sqrt{f(0)}} \left( d_0(x) + \half \kappa_0 d_0^2(x) + O(d_0^3(x)) \right)} \mbox{ as } x \to 0\, \ \mbox{ and } \
e^{-\frac{\sqrt{\mu_m}}{\sqrt{f(1)}} \left( d_1(x) + \half \kappa_1 d_1^2(x) + O(d_1^3(x)) \right)} \mbox{ as } x \to 1\,.  
$$
Using that the Steklov eigenvalues $\lambda_m^\pm$ are well approximated either by $\frac{\sqrt{\mu_m}}{\sqrt{f(0)}} + \mathcal O (1) $ or \break $\frac{\sqrt{\mu_m}}{\sqrt{f(1)}} + \mathcal O (1) $ as $m \to \infty$, we recover the results of Theorem \ref{GT} by Galkowski and Toth in \cite{GaTo2019}. 
\end{rem}

\begin{rem}
In \cite{HeKa2021} the authors analyzed the exponential decay of the eigenfunctions associated with the non positive eigenvalues $\omega_k (h)$ of the Robin problem. This corresponds to $\omega =\omega_k(h)$. Note that
 in this application $\omega_k (h)$ tends to $-\infty$ as $h\to 0$.
\end{rem}

\begin{rem}
Another natural normalization for the eigenfunctions would be  to consider the $L^2$ norm on $M$. We will show in Theorem \ref{normthm} that it changes only the estimate by  a multiplicative factor $ C\, \mu_m^\frac 14$.
\end{rem}


\subsection{Outline of the paper}

In Section \ref{ModelSteklovSpectrum}, using the symmetries of warped product manifolds, we first decompose the Dirichlet to Neumann operator $\Lambda_g$ on $(M,g)$ onto a natural family of invariant two dimensional subspaces $(\H_m)_{m\geq0}$ in Subsection \ref{Model-DNmap-2D}. This decomposition allows us to get a precise expression of the Steklov eigenvalues $\lambda_m^\pm$ on each $\H_m$ and study carefully their asymptotics and splitting $d_m = \lambda_m^+ - \lambda_m^-$ as $m \to \infty$ in the four distinct cases presented above in Subsection \ref{SteklovSpectrum}. We finally introduce the precise expressions of the Steklov eigenfunctions $\varphi_m^\pm$ in Subsection \ref{SteklovEigenfunctions}. In Section~\ref{ExponentialLocalization}, we prove our exponential localization results for the Steklov eigenfunctions. Precisely, they will obtained as consequences of more precise localization estimates given in Subsection~\ref{ExpLocI} for symmetric warped products and Subsection \ref{ExpLocII} for asymmetric warped products. Eventually, we recall in appendix \ref{WeylTitchmarsh} the necessary material on the Weyl-Titchmarsh theory of one dimensional Schr\"odinger  operators that will be needed in the course of the paper and in appendix \ref{FleaElephant} the analogy between the above problem and the flea on the elephant phenomenon for double-well potentials.  \\

\vspace{0.5cm}
\noindent \textbf{Acknowledgments}: We would like to thank Germain Gendron and Ayman Kachmar for useful discussions on the results of this paper.


\section{The Steklov eigenfunctions on warped product manifolds} \label{ModelSteklovSpectrum}

\subsection{The model} \label{Model-DNmap-2D}

Let us start with the Dirichlet problem (\ref{DirichletPb}).
Using the transformation law of the Laplace-Beltrami operator under a conformal change of the metric, we have the following convenient expression for $-\Delta_g$. 
$$
  -\Delta_g = f^{-\frac{n+2}{4}} (-\partial^2_x - \Delta_K + c_f(x)) f^{\frac{n-2}{4}}, 
$$
where $-\Delta_K$ denotes the Laplace-Beltrami operator on the transversal Riemannian manifold $(K,g_K)$ and $$c_f(x) = \frac{\left(f^{\frac{n-2}{4}}\right)''}{f^{\frac{n-2}{4}}}\,.$$
 Hence setting $v = f^{\frac{n-2}{4}} u$ and $$q_f(x) = c_f(x) - \omega f(x)\,,$$ the Dirichlet problem (\ref{DirichletPb}) becomes
\begin{equation} \label{DP1}
	\left\{ \begin{array}{rll} 
		(-\partial^2_x - \Delta_K + q_f(x)) v & = 0, & \textrm{on} \ M, \\
		v & = f^{\frac{n-2}{4}} \psi, & \textrm{on} \ \partial M.
	\end{array} \right.
\end{equation}

\begin{rem}
Note that for $n=2$, we have simply $q_f = - \omega f$. Thus $q_f = 0$ when $n=2$ and $\omega = 0$. 
\end{rem}

\begin{rem}[Geometric interpretation of $q_f$] \label{Curvature}~\\
If we denote by $S_0$ and $S$ the scalar curvatures of $g_0 = dx^2 + g_K$ and $g$ respectively, we observe that $S_0$ is equal to $S_K$ the scalar curvature of the transversal Riemannian manifold $(K,g_K)$. A consequence of the Yamabe equation \cite{Au2001}, p.171 leads for $n \geq 2$ to the following geometric interpretation of the potential $c_f$ :
$$
	c_f(x) = \frac{n-2}{4(n-1)} \left( S_K - f(x) S \right),
$$
and thus consequently
$$
  q_f(x) = \frac{n-2}{4(n-1)} \left( S_K - f(x) S \right) - \omega f(x). 
$$	
\end{rem}	

\vspace{0.3cm}
We recall now that
$$
  L^2(K, dV_K) = \bigoplus_{m \geq 0} \, \langle Y_m \rangle, 
$$
where the $(Y_m)_{m \geq 0}$ are the sequence of normalized eigenfunctions of $-\Delta_K$ associated with  the eigenvalues $(\mu_m)_{m \geq 0}$ (indexed in increasing order counting multiplicity). Define also for $j=0,1$ the functions $Y_m^j = f^{\frac{1-n}{4}}(j) Y_m$. The two families $(Y_m^j)_{m \geq 0}, \ j=0,1$ are Hilbert bases of $L^2(\Gamma_j, dV_{g_j})$ 
with $g_j = f(j) g_K$. In particular, we have
$$
L^2(\Gamma_j, dV_{g_j}) = \bigoplus_{m \geq 0} \ \langle Y_m^j \rangle, \quad j=0,1. 
$$  
Recalling that $L^2(\partial M, dV_{g_{|\partial M}}) = L^2(\Gamma_0, dV_{g_0}) \times  L^2(\Gamma_1, dV_{g_1})$, we can write 
$$
  L^2(\partial M, dV_{g_{|\partial M}}) = \bigoplus_{m \geq 0} \left( \C^2 \otimes \left( \begin{array}{c} Y_m^0 \\ Y_m^1 \end{array} \right) \right) := \bigoplus_{m \geq 0} \H_m. 
$$
This decomposition allows for separation of variables in (\ref{DP1}). Precisely, writing $v = \ds \sum_{m \geq 0} v_m(x) Y_m$ and $\psi = \ds \sum_{m \geq 0} \left( \begin{array}{c} 
\psi_m^0 \\ \psi_m^1 \end{array} \right) \otimes \left( \begin{array}{c} Y_m^0 \\ Y_m^1 \end{array} \right) $ , we see that the $v_m$'s satisfy the 1D Schr\"odinger equations 
\begin{equation} \label{SE}
	\left\{ \begin{array}{c} 
		-v_m'' + q_f(x) v_m = -\mu_m v_m\,, \quad  x \in [0,1]\,, \\
		v_m(0) = f^{-\frac{1}{4}}(0) \psi^0_m\,, \quad v_m(1) = f^{-\frac{1}{4}}(1) \psi^1_m\,. 
	\end{array} \right. 
\end{equation}
Moreover, each two-dimensional space $\H_m$ is left invariant through the action of the DN map. More precisely, for each Dirichlet data 
$$
  \psi_m = \left( \begin{array}{c} \psi_m^0 \\ \psi_m^1 \end{array} \right) \otimes \left( \begin{array}{c} Y_m^0 \\ Y_m^1 \end{array} \right) \in \H_m,
$$
the DN map simplifies as follows (we refer to \cite{Gen2020, Gen2021} and  for the explicit calculus)
\begin{equation} \label{DNm}
  \Lambda_g(\omega) \psi_m= \left[ \Lambda_g^m(\omega) \left( \begin{array}{c} \psi_m^0 \\ \psi_m^1 \end{array} \right) \right] 
  \otimes \left( \begin{array}{c} Y_m^0 \\ Y_m^1 \end{array} \right),  \quad \Lambda_g^m(\omega) = \left( \begin{array}{cc} A_m & B_m \\ B_m & C_m \end{array} \right)
\end{equation}
where
\begin{equation} \label{ACB}
  \left\{ \begin{array}{c}
  	A_m = -\frac{M(-\mu_m)}{\sqrt{f(0)}} + (n-2) \frac{f'(0)}{4f(0)^{3/2}}, \quad C_m = -\frac{N(-\mu_m)}{\sqrt{f(1)}} - (n-2) \frac{f'(1)}{4f(1)^{3/2}}, \\
  B_m = - \frac{1}{(f(0) f(1))^{1/4}} \frac{1}{\Delta(-\mu_m)}.
  \end{array} \right. 
\end{equation}
Here $M, N$ and $\Delta$ are the Weyl-Titchmarsh and characteristic functions associated with  the Schr\"odinger  equation  (\ref{SE}) with Dirichlet boundary conditions (see Appendix \ref{WeylTitchmarsh} for the definitions). Note that all the previous quantities are well-defined since $\Delta(-\mu_m) \ne 0$ for all $m \geq 0$ (see \cite{DKN2019a}, Remark 3.1). 

Let us make a few remarks before introducing the Steklov spectrum. 

\begin{rem}~\\

1. The $2 \times 2$ matrices $\Lambda_g^m(\omega), \ m \geq 0$ are symmetric reflecting the fact that the DN operator $\Lambda_g(\omega)$ is selfadjoint on $L^2(\partial M, dV_{g_{| \partial M}})$. \\

2. Clearly,  when $n \geq 3$, or when $n=2$ and $\omega \ne 0$,  we have $f(x) = f(1-x)$ for all $x \in [0,\half]$ if and only if $q_f(x) = q_f(1-x)$ for all $x \in [0,\half]$. As recalled in the appendix, under this assumption, 
we have $M(z) = N(z)$ for all $z \in \C \backslash \{poles\}$ and thus $A_m = C_m$ for all $m \geq 0$\,. \\

3. Using the universal asymptotics of the characteristic and Weyl-Titchmarsh functions recalled in Corollary \ref{AsympDE} 
and Theorem \ref{AsympWT} of the appendix, we have the following asymptotics as $m \to \infty$ :
\begin{eqnarray} 
  A_m &=& \frac{\sqrt{\mu_m}}{\sqrt{f(0)}} + \mathcal O (1), \quad C_m = \frac{\sqrt{\mu_m}}{\sqrt{f(1)}} + \mathcal O (1)\label{AsympAmCm} \\
  B_m &=& - \frac{1}{(f(0) f(1))^{1/4}} \sqrt{\mu_m}e^{-\sqrt{\mu_m}}\  \left(1 + \mathcal O  \left(\frac{1}{\sqrt{\mu_m}}\right)\right). \label{AsympBm}
\end{eqnarray}
Note that these asymptotics hold true for warping functions $f \in C^2([0,1])$ only. \\

 4. In dimension $n=2$ and for the frequency $\omega=0$, the potential $q_f(x)=0$. So, we get explicit formulae for $A_m$, $B_m$ and $C_m$, (see for instance \cite{DKN2019a}, Remark 3.1 ) : for $m \geq 1$, 
	\begin{eqnarray}
		A_m &=& \frac{1}{\sqrt{f(0)}} \ \sqrt{\mu_m} \coth (\sqrt{\mu_m}), \label{Am} \\
	    B_m &=& - \frac{1}{(f(0) f(1))^{1/4}} \  \frac{\sqrt{\mu_m}}{\sinh (\sqrt{\mu_m})},  \label{Bm} \\
		C_m &=& \frac{1}{\sqrt{f(1)}} \ \sqrt{\mu_m} \coth (\sqrt{\mu_m}),  \label{Cm}
	\end{eqnarray}
and for $m=0$, we have 	
\begin{equation}\label{A0B0C0}
	A_0 = \frac{1}{\sqrt{f(0)}} \ ,\ B_0 = - \frac{1}{(f(0) f(1))^{1/4}} \ , \	C_0 = \frac{1}{\sqrt{f(1)}}.
\end{equation}
	
\end{rem}

\subsection{The Steklov spectrum} \label{SteklovSpectrum}

The Steklov spectrum is by definition the spectrum of the DN operator $\Lambda_g(\omega)$, (we refer the reader to the nice survey \cite{GiPo2017} for the state of the art). Clearly, using the symmetry of our model, the Steklov spectrum is given by
$$
\bigcup_{m \geq 0} \ \sigma(\Lambda_g^m(\omega)),
$$
where
$$
\Lambda_g^m(\omega) = \left( \begin{array}{cc} A_m & B_m \\ B_m & C_m \end{array} \right).
$$   
The characteristic polynomial of $\Lambda_g^m(\omega)$ is
$$
P_m(\lambda) = \lambda^2 - (A_m+C_m) \lambda + (A_mC_m - B_m^2).
$$
with discriminant  
$$
\delta = (A_m - C_m)^2  + 4 B_m^2 > 0\,. 
$$
Thus for each $m \geq 0$ there are two eigenvalues
\begin{equation} \label{Eigen}
	\lambda_m^\pm = \frac{(A_m + C_m) \pm \sqrt{(A_m-C_m)^2 + 4B_m^2}}{2}. 
\end{equation}
The distance between these two eigenvalues is always given by
$$
d_m = \lambda_m^+ - \lambda_m^- = \sqrt{(A_m-C_m)^2 + 4B_m^2}\,.
$$
We can easily estimate from below the splitting between these two eigenvalues. Precisely, we always have for all $m \geq 0$, 
	\begin{equation}\label{splitting}
		d_m \geq  \frac{2 \ \mu_m}{(f(0) f(1))^{1/4}} \ \frac {1}{\sqrt{\mu_m} \sinh (\sqrt{\mu_m}) +  e^{\sqrt{\mu_m} + ||q_f||} }  \ ,
	\end{equation}
	where $||q_f||$ is the $L^2$-norm of the potential $q_f$.  Note that this estimate from below is optimal for symmetric warping function $f$. To prove this, we start with
	$$
	d_m \geq 2 \ |B_m| = \frac{2}{(f(0) f(1))^{1/4}} \  \frac{1}{| \Delta(-\mu_m) |}.
	$$
	Using (\ref{Char}), we have  $\Delta(-\mu_m)= s_0(1, -\mu_m, q)$. So, we deduce the estimate from below from (\cite{PT1987}, Theorem 3, page 13).

\begin{rem}\label{dimn=2}
		In dimension $n=2$ with $\omega =0$, we can get explicit formulae for the eigenvalues. Indeed, using (\ref{Am}) - (\ref{A0B0C0}), we obtain for $m \geq 1$,
		\begin{eqnarray}\label{vpexplicit}
			\lambda_m^{\pm} & =& \frac{\sqrt {\mu_m}} {2} \left(  \left( \frac{1}{\sqrt {f(0)}} + \frac{1}{\sqrt {f(1)}} \right) \coth (\sqrt {\mu_m}) \right. \nonumber \\ 
			&\pm& \left. \sqrt{   \left( \frac{1}{\sqrt {f(0)}} - \frac{1}{\sqrt {f(1)}} \right)^2 \coth^2 (\sqrt {\mu_m})
				+ \frac{4}{\sqrt{f(0)f(1)}   } \ \frac{1}{\sinh^2 (\sqrt {\mu_m})}  } \ 	\right),
		\end{eqnarray}
		and 
		\begin{equation}
			\lambda_0^+ =  \frac{1}{\sqrt {f(0)}} + \frac{1}{\sqrt {f(1)}} \ ,\ \lambda_0^- = 0.
		\end{equation} 
		\\
		In particular, when  $f(0) = f(1)$, we recover the results obtained in \cite{CEG2011}, (see also \cite{GiPo2017}, Example 1.3.3). Precisely, for $m \geq 1$, we have
		\begin{equation}
			\lambda_m^+ = \frac{1}{\sqrt {f(0)}} \ \sqrt {\mu_m}\coth (\frac{\sqrt {\mu_m}}{2}) \ ,\ 
			\lambda_m^- = \frac{1}{\sqrt {f(0)}} \ \sqrt {\mu_m} \tanh (\frac{\sqrt {\mu_m}}{2})
		\end{equation}
		and 
		\begin{equation}
			\lambda_0^+ =  \frac{2}{\sqrt {f(0)}}  \ ,\ \lambda_0^- = 0.
		\end{equation} 
		As a consequence, if $f(0) \not= f(1)$, we  deduce that the distance between the eigenvalues $\lambda_m^{\pm}$ is given by
        \begin{equation}
			\lambda_m^+ - \lambda_m^- = \left|  \frac{1}{\sqrt {f(0)}} - \frac{1}{\sqrt {f(1)}} \right| \sqrt{\mu_m} + 
			\mathcal O (\sqrt{\mu_m} e^{-2 \sqrt{\mu_m}}),
		\end{equation}
		and  if $f(0) = f(1)$,  one has for $m \geq 1$,
		\begin{equation}
			\lambda_m^+ - \lambda_m^- =  \frac{2}{\sqrt {f(0)}} \ \frac{\sqrt{\mu_m}}{\sinh (\sqrt{\mu_m})}.
		\end{equation}
The above explicit formulae entail immediately the asymptotic behaviors of the Steklov eigenvalues $\lambda_m^\pm$ and of the gap $d_m$ as $m \to \infty$ given in Theorem \ref{Main20}. 
\end{rem}

In order to understand the asymptotic behaviour of the eigenvalues $\lambda_m^{\pm}$ in the other cases (\textit{i.e.} when $n\geq 3$, or when $n=2$ and $\omega \ne 0$), we need to understand whether the term $(A_m-C_m)^2$ dominates, or not, the term $4B_m^2$ in the square root $\sqrt{(A_m - C_m)^2 + 4 B_m^2}$, as $m \to \infty$. For this, we shall consider the four distinct cases given in the introduction. \\

\noindent \textbf{Case I  [($\mathbf{n=2, \ \omega \ne 0}$) or ($\mathbf{n\geq3}$): symmetric warped product ]}:\\
 Let us assume $f(x) = f(1-x)$ for all $x \in [0,\half]$, or equivalently that $q_f(x) = q_f(1-x)$ for all $x \in [0,\half]$. Then for all $m \geq 0$, we have $A_m = C_m$ and thus
$$
  \lambda_m^+ = A_m + |B_m|, \quad  \lambda_m^- = A_m - |B_m|,
$$
and the two eigenvalues are exponentially closed since
$$
  d_m = 2 |B_m| =  \frac{2}{\sqrt{f(0)}} \sqrt{\mu_m} e^{-\sqrt{\mu_m}}  \left(1 + \mathcal O  \left(\frac{1}{\sqrt{\mu_m}} \right) \right).
$$
In that case, we recover the asymptotics given in (\ref{AsympAmCm}), \textit{i.e.}
$$
  \lambda_m^\pm = \frac{\sqrt{\mu_m}}{\sqrt{f(0)}} + \mathcal O (1) ,  \quad m \to \infty. 
$$

\vspace{0.3cm}
\noindent \textbf{Case II  [($\mathbf{n=2, \ \omega \ne 0}$) or ($\mathbf{n\geq3}$): asymmetric warped product ]}:\\
 Let us assume that $f = f_0 + f_1$  where $f_0$ is a symmetric warping function and $f_1$ is an asymmetric perturbation of $f_0$. Of course, the corresponding potential $q_f$ will 
 also be asymmetric in that case. We consider three subcases: \\  

\textbf{Case II.A}: We assume that the Taylor series of $f(x)$ and $f(1-x)$ differ at $0$ at the order $k$ . More precisely, let $k \geq 0$ be the smallest integer such that $f^{(k)}(0)  \ne (-1)^k f^{(k)}(1)$. \\

\noindent
Let us begin with the following lemma:

\begin{lemma} \label{AsympCaseIIA}
\begin{enumerate}
\item For $n=2$ with $\omega \not=0$, there exists a constant $a_k \not=0$ such that
\begin{eqnarray}
  A_m - C_m &=&  a_0 \ \sqrt{\mu_m} + \mathcal O  \left(  \mu_m^{-\frac{1}{2}} \right) \mbox{ if }  k=0,  \\
  A_m - C_m &=&  a_k \  \mu_m^{-\frac{k+1}{2}} + \mathcal O  \left(  \mu_m^{-\frac{k+2}{2}}   \right) \mbox{ if }  k\geq 1.
\end{eqnarray}
More precisely, we have
\begin{eqnarray}
	a_0 &=& \frac{1}{\sqrt{f(0)}} - \frac{1}{\sqrt{f(1)}}, \nonumber \\
	a_k &=& -\frac{\omega}{2^{k+1}\sqrt{f(0)} } \ \left( f^{(k)}(0) -(-1)^k f^{(k)}(1)   \right)  \mbox{ if }  k\geq 1. \nonumber
\end{eqnarray}

\item For $n \geq 3$, there exists a constant $b_k \not=0$ such that
\begin{equation}
A_m - C_m = b_k  \ \mu_m^{\frac{1-k}{2}} + \mathcal O  \left(  \mu_m^{-\frac{k}{2}} \right).
\end{equation}
More precisely, we have
$$
\left\{ \begin{array}{l}
	b_0 = \frac{1}{\sqrt{f(0)}} - \frac{1}{\sqrt{f(1)}}, \\
	b_k = \frac{n-2}{2^{k+1} f(0)^{3/2}} \ \left( f^{(k)}(0) -(-1)^k f^{(k)}(1)   \right)  \mbox{ if }  k\geq 1.
\end{array} \right. 
$$
\end{enumerate}
\end{lemma}

\vspace{0.2cm} 

\begin{proof}
1. For $n=2$, using (\ref{ACB}) we get 
\begin{equation*}
A_m = -\frac{M(-\mu_m)}{\sqrt{f(0)}}, \ C_m = -\frac{N(-\mu_m)}{\sqrt{f(1)}} . \\
\end{equation*}
By Theorem \ref{AsympWT}, the Weyl-Titchmarsh function $M$ satisfy the asymptotics
\begin{equation*} 
	M(-\kappa^2) = -\kappa - \sum_{j=0}^k \beta_j \kappa^{-j-1} +\mathcal  O(\kappa^{-k}) \mbox{ as } \kappa  \to +\infty,
\end{equation*}	 
where the constants $\beta_j$ can be computed inductively by $\beta_j = \beta_j(0)$ with
$$
\beta_0(x) = \half q_f (x), \quad \beta_{j+1}(x) = \half \beta_j'(x) + \half \sum_{\ell=0}^j \beta_\ell (x) \beta_{j-\ell}(x). 
$$
A straightforward calculation gives
$$
\beta_{j}(x) = \frac{1}{2^{j+1}} q_f^{(j)}(x) + \tilde{\beta}_j (x),
$$
where $\tilde{\beta}_j (x)$ only depends on the derivatives $q_f^{(p)}(x)$ for all $p \leq j-1$. In the same way, the asymptotics of $N(-\kappa^2)$ are given by
\begin{equation*} 
	N(-\kappa^2) = -\kappa - \sum_{j=0}^k \gamma_j \kappa^{-j-1} + \mathcal O(\kappa^{-k}) \mbox{ as } \kappa  \to +\infty\,,
\end{equation*}	 
and the constants $\gamma_j$ can be computed inductively by $\gamma_j = \gamma_j(0)$ with
$$
\gamma_0(x) = \half \check{q}_f(x), \quad \gamma_{j+1}(x) = \half \gamma_j'(x) + \half \sum_{\ell=0}^j \gamma_\ell (x) \gamma_{j-\ell}(x)\,, 
$$
where we have set $\check{q}_f(x) = q_f(1-x)$. \\ As previously, 
$$
\gamma_{j}(x) = \frac{1}{2^{j+1}} \check{q}_f^{(j)}(x) + \tilde{\gamma}_j (x)\,,
$$
where $\tilde{\gamma}_j (x)$ only depends on the derivatives $\check{q}_f^{(p)}(x)$ for all  $p \leq j-1$\,. \\

First, note that the proof in the case $k=0$ is obvious. Secondly, recalling that the potential $q_f = -\omega f$, we see that for every $ k \geq 1$, one has $\beta_p = \gamma_p$ for all  $p \leq k-1$, and the proof is complete. \\

2. For  $n \geq 3$, one has
$$ A_m = -\frac{M(-\mu_m)}{\sqrt{f(0)}} + (n-2) \frac{f'(0)}{4f(0)^{3/2}}, \ C_m = -\frac{N(-\mu_m)}{\sqrt{f(1)}} - (n-2) \frac{f'(1)}{4f(1)^{3/2}}, 
$$
and 
$$
q_f (x) = \frac{(n-2)}{4} \frac{f''(x)}{f(x)} + \frac{(n-2)}{4} \  \frac{(n-6)}{4} \left(\frac{f'(x)}{f(x)}\right)^2 .
$$
Then, we follow the same strategy as in the case $n=2$. We leave the details to the reader.
\end{proof}

\vspace{0.3cm}

Since $B_m^2 = \mathcal O(\mu_m e^{-2\sqrt{\mu_m}})$, we see that in all dimensions the term $(A_m - C_m)^2$ dominates the term $4 B_m^2$  as $m \to \infty$. However, to distinguish the two dimensional case and the case $n \geq 3$, we need to introduce  the  sequence $(\alpha_k)_{k \geq 0}$ defined by  $\alpha_0 = \half$ and for $k \geq 1$,

\begin{equation} \label{alphak}
\alpha_k = 	\left\{ \begin{array}{cl}
	  \frac{k+3}{2} \ \mbox{ if }  \ n=2, \ \omega \not= 0. \\
	 \frac{k+1}{2} 	\  \mbox{ if } \ n\geq 3.\hspace{1cm} 
	\end{array} \right.	
\end{equation}

\noindent
Coming back to (\ref{Eigen}) we thus obtain  :

\begin{prop}\label{eigensym}
	Under the above asumptions, we get :
\begin{enumerate}
		\item For $n=2$ with $\omega \not=0$, if  $a_k >0$, (resp. if $b_k >0$ for $n\geq 3$), then
		\begin{eqnarray}
			\lambda_m^+ &=&  A_m +  \mathcal O \ \left( \mu_m^{\alpha_k} \ e^{-2\sqrt{\mu_m}}\right)\mbox{ as } m \to  \infty, \\
			\lambda_m^- &=&  C_m + \mathcal  O \ \left( \mu_m^{\alpha_k}  \ e^{-2\sqrt{\mu_m}}\right)\mbox{ as } m \to  \infty, 
		\end{eqnarray}
	
		\item For $n=2$ with $\omega \not=0$, if  $a_k <0$, (resp. if $b_k <0$ in dimension $n\geq 3$), then
		\begin{eqnarray}
			\lambda_m^+ &=&  C_m +  \mathcal O \ \left( \mu_m^{\alpha_k} \ e^{-2\sqrt{\mu_m}}\right)\mbox{ as } m \to  \infty, \\
			\lambda_m^- &=&  A_m + \mathcal  O \ \left( \mu_m^{\alpha_k} \ e^{-2\sqrt{\mu_m}}\right)\mbox{ as } m \to  \infty.
		\end{eqnarray}
\end{enumerate}
	Depending on the sign of  $a_k$ (resp. $b_k)$, we see that the eigenvalues $\lambda_m^\pm$ are approximated by either $A_m$, or $C_m$, whose asymptotics are given in (\ref{AsympAmCm}). Moreover, the distance between the two eigenvalues $\lambda_m^\pm$ satisfies for $n=2$ with $\omega \not=0$,
	\begin{eqnarray}\label{dmsym1}
		d_m &=&  |a_0| \ \sqrt{\mu_m} +  \mathcal O \ \left(  1 \right)   \mbox{ if } \ k=0, \\
        d_m &=& |a_k| \ \mu_m^{-\frac{1+k}{2}} +  \mathcal O \ \left(  \mu_m^{-\frac{k+2}{2}}\right) \mbox{ if } \ k \geq 1 \, ,
\end{eqnarray}
whereas for  $n \geq 3$,
\begin{equation}\label{dmsym3}
	d_m =  |b_k| \ \mu_m^{\frac{1-k}{2}} +  \mathcal O \ \left(  \mu_m^{-\frac{k}{2}}\right)   \mbox{ if  }  k \geq 0\,.
\end{equation}

\end{prop}

\vspace{0.3cm}
\textbf{Case II.B}: Now, we assume that there exists $a \in [0,\half[$ such that $f(x) = f(1-x)$ for all $x \in [0,a]$ and if $a=0$, we only assume that the Taylor series of $f(x)$ and $f(1-x)$ are equal at $x=0$ at all orders. Moreover, we make the following sign assumption: there exists $ \delta \in ]0,\half-a[$ \ such that for all $ x \in ]a,a+\delta]$, either
\begin{equation} \label{Hyp-q1}
  f(1-x) \left( \frac{n-2}{4(n-1)} S(1-x) + \omega  \right) > f(x) \left( \frac{n-2}{4(n-1)} S(x) + \omega  \right), 
\end{equation}
or 
\begin{equation} \label{Hyp-q2}
	f(1-x) \left( \frac{n-2}{4(n-1)} S(1-x) + \omega  \right) < f(x) \left( \frac{n-2}{4(n-1)} S(x) + \omega  \right),
\end{equation}
where $S(x)$ is the scalar curvature of $(M,g)$ at $x$.
 
\begin{rem}\label{cepsilon} 
\begin{enumerate}
	\item In the case where $a >0$, since $f$ is smooth, the Taylor series of $f(x)$ and $f(1-x)$ also coincide at $x=a$ at all orders, (and the same is true for $q_f(x)$ and $q_f(1-x)$). 
\item We can see that, if $n=2$, the condition (\ref{Hyp-q1}) (for instance) is thus equivalent to 
$$
  f(1-x) > f(x),
$$  
and if $n \geq 3$ and $\omega = 0$, this is equivalent to 
$$
  f(1-x) S(1-x) > f(x) S(x).  
$$
\item According to Remark \ref{Curvature}, (\ref{Hyp-q1}) is equivalent to the simplest condition $q_f(x) - q_f(1-x) > 0$ on $]a,a+\delta]$, whereas (\ref{Hyp-q2}) is equivalent to $q_f(x) - q_f(1-x) < 0$ on the same interval.
This implies in particular that for all $0 < \e < \delta$, there exists a constant $ c_\e > 0$ such that for all $x \in [a+ \e, a+ \delta ]$, $q_f(x) - q_f(1-x) \geq c_\e$ or $q_f(x) - q_f(1-x) \leq -c_\e$. 
\end{enumerate}
\end{rem}

Let us prove the lower bound:
\vspace{0.2cm}

\begin{prop} \label{EstimateBelow}
Under the above assumptions, for all $\e >0$ small enough there exist a positive constant $c_{\e}$ and an integer $m_\e$ large enough such that for all $m \geq m_{\e}$ ,
$$
  |A_m - C_m| \geq c_\e \  e^{-2(a+\e)\sqrt{\mu_m}}. 
$$
\end{prop} 
\begin{proof}
Set $L(x) = q_f(x) - q_f(1-x)$ and assume that $L(x) = 0$ for all $x \in [0,a]$ and $L(x) > 0$ for all $x \in ]a,a+\delta]$. We start from the general formula proved in Proposition \ref{lemmaLinkMN} :

\begin{equation} 
  M_q(z) - N_q(z) = \int_0^1 L(x) \Psi(x,z,q) \Phi(1-x,z,q) dx.
\end{equation}
From Lemma \ref{AsympFSS} and Corollary \ref{AsympDE}, the asymptotics 
\begin{align}
	\Psi(x,-\mu_m,q) & = \frac{s_1(x,-\mu_m,q)}{\Delta_q(-\mu_m)}  = -2e^{-\sqrt{\mu_m}} \sinh(\sqrt{\mu_m}(1-x)) + \mathcal O \left( \frac{e^{-\sqrt{\mu_m} x}}{\sqrt{\mu_m}} \right)\mbox{ as } m \to  \infty, \label{AsympPsi} \\
	\Phi(x,-\mu_m,q) & = \frac{s_0(x,-\mu_m,q)}{\Delta_q(-\mu_m)}  = 2e^{-\sqrt{\mu_m}} \sinh(\sqrt{\mu_m}x) + \mathcal O\left( \frac{e^{-\sqrt{\mu_m}(1-x) }}{\sqrt{\mu_m}} \right)\mbox{ as } m \to  \infty, \label{AsympPhi}
\end{align}
hold uniformly for $x \in [0,1]$. \\
Thus in particular, we have
\begin{align}
	\Psi(x,-\mu_m,q) &  = e^{-\sqrt{\mu_m}x} \left(1 - e^{-2\sqrt{\mu_m}(1-x)} + O(\frac{1}{\sqrt{\mu_m}}) \right) =  \mathcal O \left(e^{-\sqrt{\mu_m} x} \right)\mbox{ as } m \to  \infty, \label{Psi} \\
	\Phi(x,-\mu_m,q) & = e^{-\sqrt{\mu_m}(1-x)} \left(1 - e^{-2\sqrt{\mu_m}x} + O(\frac{1}{\sqrt{\mu_m}}) \right) =  \mathcal O \left(e^{-\sqrt{\mu_m}(1-x)} \right)\mbox{ as } m \to  \infty, \label{Phi}
\end{align}
uniformly in $x \in [0,1]$.
Then we write 
\begin{align*}
  M_q(-\mu_m) - N_q(-\mu_m) & = \int_a^{a+\delta} L(x) \Psi(x,-\mu_m,q) \Phi(1-x,-\mu_m,q) dx \\ 
  & \hspace{2cm} + \int_{a+\delta}^{1-a} L(x) \Psi(x,-\mu_m,q) \Phi(1-x,-\mu_m,q) dx.  
\end{align*} 
On one hand, observe that using (\ref{Psi})-(\ref{Phi}), the second integral can be estimated by
$$
  \int_{a+\delta}^{1-a} L(x) \Psi(x,-\mu_m,q) \Phi(1-x,-\mu_m,q) dx = \mathcal O\left( \frac{e^{-2(a+\delta) \sqrt{\mu_m}}}{\sqrt{\mu_m}} \right)\mbox{ as } m \to  \infty.  
$$
On the other hand, using the assumptions on $L$ and (\ref{Psi})-(\ref{Phi}) again, the first integral can be estimated from below as follows. For all $0 < \e < \delta$ and as $m\rightarrow +\infty $ 
\begin{align*}
  \int_{a}^{a+\delta} L(x) \Psi(x,-\mu_m,q) \Phi(1-x,-\mu_m,q) dx 
  & \geq \half \int_{a}^{a+\delta} L(x) e^{-2\sqrt{\mu_m}x} dx \\ 
  & \geq \frac{c_\e}{2} \int_{a+\e}^{a+\delta} e^{-2\sqrt{\mu_m}x} dx  \\
  & \geq  \frac{c_\e}{4}  \frac{e^{-2(a+\e) \sqrt{\mu_m}}}{\sqrt{\mu_m}} + \mathcal O\left( \frac{e^{-2(a+\delta) \sqrt{\mu_m}}}{\sqrt{\mu_m}}  \right), 
\end{align*}
where the constant $c_{\e}$ is given in Remark \ref{cepsilon}.
Putting everything together, we see that for all $\e >0$ small enough, there exist $c_{\e}>0$ and $m_{\e} >0 $  such that for $m \geq m_{\e}$,
$$
   M_q(-\mu_m) - N_q(-\mu_m) \geq c_\e \ {\sqrt{f(0)}} \ e^{-2(a+\e) \sqrt{\mu_m}}\,.
$$
Since 
$$
  A_m - C_m = \frac{N_q(-\mu_m) - M_q(-\mu_m)}{\sqrt{f(0)}},
$$
we have 
$$
  | A_m - C_m | = \frac{M_q(-\mu_m) - N_q(-\mu_m)}{\sqrt{f(0)}} \geq c_\e \ e^{-2(a+\e) \sqrt{\mu_m}}\mbox{ as } m \to  \infty, 
$$
and the result is proved. The other case can be treated similarly. 
\end{proof}

\begin{rem}\label{BML}
Under the above assumptions on $f$, the local Borg-Marchenko theorem \ref{BMlocal} entails in the case $a>0$ that for all $\e > 0$, there exist  $C_\e > 0$ and $m_\epsilon$ such that for  $m\geq m_\epsilon\,$, 
$$
  | A_m - C_m | \leq C_\e\  e^{-2(a-\e)\sqrt{\mu_m}},	 
$$		
whereas if $a=0$, Corollary \ref{orderinfty} implies
$$
| A_m - C_m | = \mathcal O \left(\mu_m^{-\infty}\right).
$$
\end{rem}	

We deduce from Lemma \ref{EstimateBelow} that the term $(A_m - C_m)^2$ dominates the term $4B_m^2$ as $m \to \infty$. 
Coming back to (\ref{Eigen}), we get (similarly to the case II.A.) for all $\e >0$,
$$
  \lambda_m^\pm = \frac{(A_m + C_m) \pm |A_m - C_m|}{2} + \mathcal O\left(e^{-2(1-a-\epsilon) \sqrt{\mu_m}}\right)\mbox{ as } m \to  \infty.
$$
Even more precisely, we have the following result:

\begin{prop}\label{evcaseII.B}
If the condition (\ref{Hyp-q1}) is satisfied, then for all $\e >0$, there exist $c_{\e}>0$ and $m_{\e}$  such that for all $m \geq m_{\e}$ 
\begin{eqnarray}
   \lambda_m^+ &=&  C_m +   \mathcal O\left(e^{-2(1-a-\epsilon) \sqrt{\mu_m}}\right), \\
     \lambda_m^- &=&  A_m +  \mathcal O\left(  e^{-2(1-a-\epsilon) \sqrt{\mu_m}}\right),
\end{eqnarray}   
whereas, if (\ref{Hyp-q2}) holds then, 
\begin{eqnarray}
	\lambda_m^+ &=&  A_m +  \mathcal O\left( e^{-2(1-a-\epsilon) \sqrt{\mu_m}}\right), \\
	\lambda_m^- &=&  C_m +  \mathcal O\left( e^{-2(1-a-\epsilon) \sqrt{\mu_m}}\right),
\end{eqnarray}   
At last, the gap between the two eigenvalues $\lambda_m^\pm$ satisfies the lower bound for all $m \geq m_{\e}$: 
\begin{equation}\label{dmcaseII.B}
d_m = \sqrt{(A_m-C_m)^2 + 4B_m^2}   \geq c_\e \ e^{-2(a+\e) \sqrt{\mu_m}}.
\end{equation}	
\end{prop}

\vspace{0.3cm}
\textbf{Case II.C}: It remains to study the general case $f = f_0 + f_1$  where $f_0$ is a symmetric warping function and $f_1$ is an asymmetric perturbation of $f_0$. Here, we only assume that the Taylor series of $f(x)$ and $f(1-x)$ are equal at $x=0$ at all orders. As previously, we need to understand whether the term $(A_m-C_m)^2$ dominates, or not, the term $4B_m^2$ in (\ref{Eigen}) as $m$ tends to infinity. We shall obtain a  result relatively close to the case II.B, but since we are not able to find a precise asymptotic expansion (or a bound from below) of $M_q(-\mu_m) - N_q(-\mu_m)$ as $m \to \infty$, we do not obtain a quantitative result.

\vspace{0.3cm}
Nevertheless, we have the following proposition :

\begin{prop} \label{SubsequenceIIC}
	There exists a subsequence $(m_k)_{k \geq 0}$ such that the quantity $\left| \frac{A_{m_k} - C_{m_k}}{B_{m_k}} \right| $ tends to infinity as $k \to \infty$.
\end{prop}

\begin{proof}
We make a proof by contradiction and we assume that the sequence $\left( |\frac{A_{m} - C_{m}}{B_{m}}| \right)$ is bounded. Clearly, thanks to (\ref{ACB}) and  (\ref{MN}), on has: 
$$
 |\frac{A_{m} - C_{m}}{B_{m}}| = | (M_q (-\mu_m) -  N_q (-\mu_m)) \Delta_q (-\mu_m) | = | D_q (-\mu_m) -  E_q (-\mu_m) |.
 $$
Thus, there exists a constant $C>0$ such that
$$
| D_q (-\mu_m) -  E_q (-\mu_m) | \leq C \ ,\  \forall m \geq 0.
$$
Let us introduce the function $F(z) := D_q(-z^2) -E_q(-z^2)$. Thanks to Corollary \ref{AsympDE},  $F(z)$ is an entire function  satisfying the estimate:
$$
|F(z)| \leq  C \ e^{Re \ z } \ \ , \ \ {\rm{for}} \  Re \  z \geq 0.
$$ 
Moreover, by assumption, $F(z)$ is  bounded on the sequence $(\sqrt{\mu_m})$.
We recall that the Weyl law implies the following asymptotics on the $\sqrt{\mu_m}$ (repeated according multiplicity):
$$
\sqrt{\mu_m} =  c_n \  m^{\frac{1}{(n-1)}} + \mathcal O (1) ,
$$
where $c_n$ denotes a suitable constant independent of $m$. Setting  $\lambda_m = \frac{1}{c_n} \sqrt{\mu_{m^{n-1}}}$, we deduce  there exists $c>0$ such that $|\lambda_m - m | \leq  c$. Now, for a fixed $N \in \N$ large enough, we set $\nu_m = \frac{\lambda_{mN}}{N}$, and we have $|\nu_m -m | \leq \frac{c}{N}<\frac{1}{4}$. We introduce a new function $G(z)= F(c_n N z )$, 
which satisfies the same properties as $F(z)$ and is bounded on the sequence $(\nu_m)$. It follows from a theorem of Duffin and Schaeffer (\cite{Boa1954}, Theorem 10.5.1) that $G(x)$, (and also $F(x)$), is bounded for $x>0$.
Clearly, it implies that, as $\kappa \to +\infty$,
$$
M_q(-\kappa^2) = N_q(-\kappa^2) + \tilde{\mathcal O}(e^{-\kappa}),
$$
where $f(\kappa) =  \tilde{\mathcal O}(e^{-\kappa})$ means that for all $\e >0$, $f(\kappa) =  \mathcal O (e^{-\kappa (1-\e)})$ as $\kappa \to +\infty$.
Using the local Borg-Marchenko's theorem (see Theorem \ref{BMlocal}) and (\ref{Sym}), we see that $q(x) = q(1-x)$ for all $x \in [0,\frac{1}{2}]$, but this is not possible since $f_1(x)$ is an asymmetric perturbation of $f_0(x)$.
\end{proof}

\vspace{0.3cm}
As a by-product, there exists a subsequence, that we shall still denote by  $(m_k)_{k \geq 0}$, such that $ \frac{A_{m_k} - C_{m_k}}{B_{m_k}} \to + \infty$ or $- \infty$, as $k \to \infty$. Thus, using the same proof as in Case II.B and Remark \ref{BML}, we can state:

\begin{prop}\label{eigenvgeneral}
\begin{enumerate}
\item If we assume that 
$$
	\frac{A_{m_k} - C_{m_k}}{B_{m_k}} \to - \infty,
$$
then, for all $p \geq 0$
$$
	\lambda_{m_k}^+ =  A_{m_k} + \mathcal O\left(\mu_{m_k}^{-p} \right), \quad \lambda_{m_k}^- =  C_{m_k}+ \mathcal O\left(  \mu_{m_k}^{-p}       \right).
$$
\item If we assume that 
$$
	\frac{A_{m_k} - C_{m_k}}{B_{m_k}} \to + \infty\,,
$$
then,  for all $p \geq 0$
$$
	\lambda_{m_k}^+ =  C_{m_k} + \mathcal O\left(  \mu_{m_k}^{-p}     \right), \quad \lambda_{m_k}^- =  A_{m_k}+ \mathcal O\left( \mu_{m_k}^{-p} \right).
$$
\end{enumerate}	
\end{prop}

\vspace{0.2cm}
\begin{rem}
 \begin{enumerate}	
	 \item If we assume that that there exists $a \in ]0,\half[$ such that $f(x) = f(1-x)$ for all $x \in [0,a]$, without the sign asumptions as in the case II.B,  we can improve the previous estimates thanks to the local Borg-Marchenko theorem, (see Remark \ref{BML}) . For instance, if we consider the case
	 $$
	 \frac{A_{m_k} - C_{m_k}}{B_{m_k}} \to - \infty,
	 $$
	 then we have for all $\epsilon >0$,
	 $$
	 \lambda_{m_k}^+ =  A_{m_k} + \mathcal O\left(  e^{-2(a-\epsilon) \sqrt{\mu_{m_k)}    }}     \right), \quad \lambda_{m_k}^- =  A_{m_k}+ \mathcal O\left(  e^{-2(a-\epsilon) \sqrt{\mu_{m_k)}    }}  \right).
	 $$
	 \item Note also that in the general case II.C, we are not able to find a precise (sharp) lower bound for $d_m$, (even for $d_{m_k}$), as in the cases II.A and II.B. 
	\end{enumerate}
\end{rem}


\subsection{The Steklov eigenfunctions} \label{SteklovEigenfunctions}

The Steklov eigenfunctions are now defined as the $\omega$-harmonic extensions of the eigenfunctions of the DN operator $\Lambda_g(\lambda)$. Let us calculate the latters first. For each $m \geq 0$, the eigenspaces associated with  the eigenfunctions $\lambda_m^\pm$ of $\Lambda_g^m(\omega)$ are given by
$$
E_m^\pm = \{ (a_m^\pm,b_m^\pm) \in \C^2 / (A_m - \lambda_m^\pm) a_m^\pm + B_m b_m^\pm = 0 \}. 
$$
Since $B_m \ne 0$ for all $m \geq 0$, let us choose 
$$
a_m^\pm = 1, \quad b_m^\pm = \frac{\lambda_m^\pm - A_m}{B_m}. 
$$
We deduce that for all $m \geq 0$
$$
\psi_m^\pm = \left( \begin{array}{c} 1 \\ \frac{\lambda_m^\pm - A_m}{B_m} \end{array} \right) \otimes \left( \begin{array}{c} Y_m^0 \\ Y_m^1 \end{array} \right),
$$
are (non-normalized) eigenfunctions of $\Lambda_g(\lambda)$ associated with  the eigenvalues $\lambda_m^\pm$. Since
\begin{equation} \label{nm}
  \| \psi_m^\pm \|_{L^2(\partial M)} = \sqrt{ 1 + \left( \frac{\lambda_m^\pm - A_m}{B_m} \right)^2 } =: n_m^\pm,
\end{equation}
a sequence of normalized eigenfunctions of $\Lambda_g(\lambda)$ are given by 
$$
  \phi_m^\pm = \frac{\psi_m^\pm}{n_m^\pm} =  \left( \begin{array}{c} \frac{1}{n_m^\pm} \\ \frac{\lambda_m^\pm - A_m}{B_m n_m^\pm} \end{array} \right) \otimes \left( \begin{array}{c} Y_m^0 \\ Y_m^1 \end{array} \right). 
$$

Coming back to (\ref{SE}), we consider now the functions $v_m^\pm$ solutions of the boundary value problems
$$
\left\{ \begin{array}{c} 
-(v_m^\pm)'' + q_f(x) v_m^\pm = -\mu_m v_m^\pm, \quad  x \in [0,1], \\
v_m^\pm(0) = f^{-\frac{1}{4}}(0), \quad v_m^\pm(1) = f^{-\frac{1}{4}}(1) \frac{ \lambda_m^\pm - A_m}{B_m}. 
\end{array} \right. 
$$
Recalling that the Weyl solutions $\{ \Psi(x,-\mu_m,q), \Phi(x,-\mu_m,q)\}$ are a FSS of (\ref{SE}), we look for $v_m^\pm$ under the form
$$
v_m^\pm(x) = \alpha_m^\pm \Psi(x,-\mu_m,q) + \beta_m^\pm \Phi(x,-\mu_m,q).
$$
The defining properties of the Weyl solutions immediately lead  to
\begin{equation} \label{vm}
	v_m^\pm(x) =  \frac{\Psi(x,-\mu_m,q)}{f^{\frac{1}{4}}(0)} + \frac{\lambda_m^\pm - A_m}{B_m} \frac{\Phi(x,-\mu_m,q)}{f^{\frac{1}{4}}(1)}. 
\end{equation}
Thus, the (non-normalized) Steklov eigenfunctions have  the following expression
\begin{equation} \label{SteklovEigen} 
	u_m^\pm(x,\theta)  = f^{\frac{2-n}{4}}(x) v_m^\pm(x) Y_m(\theta), \quad \forall m \geq 0. 
\end{equation}

In the next section, we shall study the localization of the Steklov eigenfunctions (\ref{SteklovEigen}) normalized in two different but natural ways. First, we normalize them by demanding that their $L^2$ norm on $\partial M$ is equal to $1$. We introduce thus
\begin{equation} \label{NormalizedSteklovEigenfunction2}
	\varphi_m^\pm(x,\theta) = \frac{u_m^\pm(x,\theta)}{n_m^\pm} =  f^{\frac{2-n}{4}}(x) w_m^\pm(x) Y_m(\theta), \quad \forall m \geq 0, 
\end{equation}
where 
\begin{equation} \label{wm}
	w_m^\pm(x) = \frac{1}{n_m^\pm} \frac{\Psi(x,-\mu_m,q)}{f^{\frac{1}{4}}(0)} + \frac{\lambda_m^\pm - A_m}{B_m n_m^\pm} \frac{\Phi(x,-\mu_m,q)}{f^{\frac{1}{4}}(1)}. 
\end{equation}
Secondly, we normalize them by demanding that their $L^2$ norm on $M$ is equal to $1$, \textit{i.e.}
\begin{equation} \label{NormalizedSteklovEigenfunction}
	\tilde{\varphi}_m^\pm (x,\theta)= \frac{u_m^\pm(x,\theta)}{\| u_m^\pm \|_2}, \quad \| u_m^\pm \|_2 = \left( \int_0^1 |v_m^\pm(x)|^2 f(x) dx \right)^{\half}, \quad \forall m \geq 0.  
\end{equation}	
Note that the first normalization $\varphi_m^\pm$ corresponds to the normalization used in \cite{HiLu2001, GaTo2019}.

\begin{rem}
When $n=2$ and $\omega=0$, we can get explicit expressions for the eigenfunctions $\varphi_m^\pm(x,\theta)$ thanks to   (\cite{DKN2019a}, Remark 3.1). For instance, if  $f(0)=f(1)$, one has for $m \geq 1$,
\begin{align*} \label{NormalizedSteklovEigenfunctiondim2}
	\varphi_m^+ (x,\theta)& = \frac{\sqrt 2}{ f^{\frac{1}{4}}(0)}  \coth(\sqrt{\mu_m}) \ \sinh(\sqrt{\mu_m} (1-2x)) \  Y_m(\theta), \\ 
	\varphi_m^- (x,\theta)& = \frac{\sqrt 2}{ f^{\frac{1}{4}}   (0)   } \cosh(\sqrt{\mu_m} (1-2x)) \  Y_m(\theta). 
\end{align*}	
But, if $f(0) < f(1)$, we have
$$
\varphi_m^\pm(x,\theta)  = \frac{1}{f^{\frac{1}{4}}(0) \sqrt{1 + (\alpha_m^\pm)^2}} \frac{\sinh(\sqrt{\mu_m}(1-x))}{\sinh(\sqrt{\mu_m})} + \frac{1}{f^{\frac{1}{4}}(1)} \sqrt{ \frac{(\alpha_m^\pm)^2}{1 + (\alpha_m^\pm)^2 }} \frac{\sinh(\sqrt{\mu_m}x)}{\sinh(\sqrt{\mu_m})}, 
$$
where
$$
  \alpha_m^\pm = \half \frac{ \sqrt{f(1)} - \sqrt{f(0)}}{(f(0) f(1))^\frac{1}{4}} \cosh(\sqrt{\mu_m}) \left( 1 \mp \sqrt{ 1 + \frac{4 \sqrt{f(0) f(1)}}{(\sqrt{f(1)} - \sqrt{f(0)})^2} \frac{1}{\cosh^2(\sqrt{\mu_m})} }    \right).
$$
If $f(0) > f(1)$, then the same formula hold with $+$ and $-$ inverted. 

The above explicit formula allow us to prove the localization results on the Steklov eigenfunctions $\varphi_m^\pm$ given in Theorem \ref{Main20} in the case $n=2$ and $\omega = 0$. 
\end{rem}


\section{Exponential localization of the Steklov eigenfunctions: the flea on the elephant phenomenon} \label{ExponentialLocalization}

In this section, we assume ($n=2, \ \omega \ne 0$) or ($n \geq 3$) and  we show that the normalized Steklov eigenfunctions corresponding to the warped product $(M,g)$ are exponentially localized at the boundary $\partial M$ as $m \to \infty$, proving in that particular case the results of Hislop, Lutzer  \cite{HiLu2001} and Galkowski, Toth \cite{GaTo2019}. More precisely  we prove that the Steklov eigenfunctions corresponding to a \emph{symmetric} warped product are exponentially localized at both boundaries $x=0$ and $x=1$ as $m \to \infty$. On the contrary, if the warped product is \emph{asymmetric}, then  we roughly speaking prove that half the Steklov eigenfunctions are localized at $x=0$, whereas the other half Steklov eigenfunctions are localized at $x=1$ as $m \to \infty$. This result is similar to what Barry Simon calls the \emph{flea on the elephant phenomenon} for Schr\"odinger  operators with an asymmetric double well potential (see \cite{HeSj1985, Si1985}).

\subsection{On H\"ormander's $L^\infty$ estimates.}

Using the H\"{o}rmander's $L^\infty$ bound in the version given by Donnelly \cite{Do2001},  we can consider any transversal compact Riemannian manifolds $(K,g_K)$ with metric coefficients in $C^2$ and have:
\begin{prop}
There exists a constant $C >0$ such that for any $\mu \in \sigma(-\Delta_K)$ and any $L^2$-normalized eigenfunction $Y_\mu$, we have
\begin{equation} \label{HormanderBound}
	||Y_\mu||_{L^\infty(K)}  \leq C \sqrt{\mu}^{\frac{n-2}{2}} \,.
\end{equation}
Here the constant $C$ only depends on a bound for the absolute value of the sectional curvature and a lower bound for the injectivity radius on $K$. 
\end{prop}
Using this proposition, we can thus focus on the $x$-dependence of the Steklov eigenfunctions  in what follows.


\subsection{The case of symmetric warping functions} \label{ExpLocI}

We prove here that for symmetric warped products, the normalized Steklov eigenfunctions are exponentially localized and equi-distributed at both boundaries $x=0$ and $x=1$ as $m \to \infty$.  Recalling that 
$$
	\varphi_m^\pm (x, \theta) =  f^{\frac{2-n}{4}}(x) w_m^\pm(x) Y_m (\theta)\,, 
$$
we give first the asymptotics of $	w_m^\pm(x)$ when $m \to \infty$.

\begin{prop} \label{SteklovLocalizationSymmetric}
	Assume that the warping function is symmetric, that is $f(x) = f(1-x)$ for all $x \in [0,\half]$. Then, when $m \to \infty$ and uniformly for $x \in [0,1]$,
$$
	w_m^\pm(x)  = -  \frac{\sqrt{2}}{ f^{\frac{1}{4}} (0)}  \ e^{-\sqrt{\mu_m} }   
	\big( \sinh(\sqrt{\mu_m}(1-x)) \pm \sinh (\sqrt{\mu_m} x)  \big) + \mathcal O (\frac{e^{-\sqrt{\mu_m }x}}{\sqrt{\mu_m} })  
	+ \mathcal O (\frac{e^{-\sqrt{\mu_m} (1-x)}}{\sqrt{\mu_m} }) .
$$
\end{prop}

\begin{proof}
Under the assumption of symmetry of $f$, we know that  $A_m = C_m$ and $\lambda_m^\pm = A_m \pm |B_m|$ for all $m \geq 0$. Consequently, using (\ref{nm}), we have
$$
  \frac{\lambda_m^\pm - A_m}{B_m} = \pm \frac{|B_m|}{B_m} = \mp 1, \quad n_m^\pm = \sqrt{2}.
$$
Hence, we  obtain
$$
  w_m^\pm(x) =  \frac{ \Psi(x,-\mu_m,q) \mp \Phi(x,-\mu_m,q)}{\sqrt{2} \ f^{\frac{1}{4}}(0)}.
$$
Then, the result follows immediately from (\ref{AsympPsi}) and (\ref{AsympPhi}). 
\end{proof}

As a by-product, we obtain:

\begin{coro} \label{SteklovLocalizationSymmetricpointwise}
	Assume that the warping function is symmetric, that is $f(x) = f(1-x)$ for all $x \in [0,\half]$. Then, there exist $C>0$ and $m_0>0$ such that for all $m\geq m_0$ and for all $(x, \theta) \in [0,1] \times K$
	$$
	|\varphi_m^\pm(x,\theta)| \leq C  \left( e^{-\sqrt{\mu_m} x} + e^{-\sqrt{\mu_m} (1-x)} \right) |Y_m(\theta)|.
	$$ 
\end{coro}

The pointwise estimates for symmetric warped products given in Corollary \ref{SteklovLocalizationSymmetricpointwise} improve the corresponding estimates obtained for flat warped product ($f(x)=1$ in our notations) given by Galkowski and Toth in \cite{GaTo2019}, example 1.1.2.

\subsection{The case of asymmetric warping functions} \label{ExpLocII}

We assume here that $f(x) = f_0(x) + f_1(x)$ where $f_0$ is symmetric with respect to $\half$ and $f_1$ is an asymmetric perturbation of $f_0$. \\

\textbf{Case II.A}. As previously we need to distinguish the two dimensional case and the case $n\geq 3$. For $k \geq 0$, we set
\begin{equation} \label{betak}
	\beta_k = 	\left\{ \begin{array}{cl}
		a_k \ \mbox{if} \ n=2, \ \omega \not= 0. \\
		b_k 	\  \mbox{if} \ n\geq 3,\hspace{1.1cm} 
	\end{array} \right.	
\end{equation}	
where $a_k$ and $b_k$ are the constants given in Lemma \ref{AsympCaseIIA}. \\

\begin{prop} \label{LocalizationIIA}
	If $k \geq 0$ is the smallest integer such that $f^{(k)}(0)  \ne (-1)^k f^{(k)}(1)$ and $\beta_k$  is negative, then
		\begin{align}
		w_m^+(x) = & -\frac{2}{\beta_k f(0)^{\half} f(1)^{\frac{1}{4}}} \ \mu_m^{\alpha_k-\half}  \ e^{-2\sqrt{\mu_m}} 
		\sinh(\sqrt{\mu_m}(1-x) ) -\frac{2}{ f(1)^{\frac{1}{4}}} e^{-\sqrt{\mu_m}} 	\sinh(\sqrt{\mu_m}x ) \nonumber \\ 
	 	           & \qquad + \mathcal O \left( \mu_m^{\alpha_k - 1} e^{-\sqrt{\mu_m }(x+1)} \right)  
		             + \mathcal O \left(\frac{e^{-\sqrt{\mu_m} (1-x)}}{\sqrt{\mu_m} }\right),\label{LocalizationIIA+} \\
		w_m^-(x) = & -\frac{2}{ f(0)^{\frac{1}{4}}} e^{-\sqrt{\mu_m}} 	\sinh(\sqrt{\mu_m}(1-x) ) + \frac{2}{\beta_k f(0)^{\frac{1}{4}} f(1)^{\half}}   \ \mu_m^{\alpha_k -\half}  \    e^{-2\sqrt{\mu_m}} \sinh(\sqrt{\mu_m}x ) 
		 \nonumber \\ 
		& \qquad + \mathcal O \left( \mu_m^{\alpha_k -1} e^{-\sqrt{\mu_m }(2-x)} \right)  
		+ \mathcal O \left(\frac{e^{-\sqrt{\mu_m} x}}{\sqrt{\mu_m} }\right),\label{LocalizationIIA-} 
		\end{align}
	holds uniformly for $x \in [0,1]$, where $\alpha_k$ is the constant  introduced in (\ref{alphak}). Moreover, the same asymptotics hold with $+$ and $-$ inverted if $\beta_k$ is positive.
\end{prop}

\begin{proof}
We only give the proof in the case $n \geq 3$ since the two dimensional case is similar. We need to study the term
$$
  \frac{\lambda_m^\pm - A_m}{B_m},
$$
that appears in the expressions (\ref{nm}) or (\ref{wm}) of $n_m^\pm$ and $w_m^\pm$ respectively. Observe first that we can rewrite this term as
\begin{equation} \label{b1}
  \frac{\lambda_m^\pm - A_m}{B_m} = \half \left[ \frac{C_m-A_m}{B_m} \mp \sqrt{ \left( \frac{C_m-A_m}{B_m} \right)^2 + 4} \right], 
\end{equation}
since for $m$ large enough $B_m$ is negative . Using the notations of Lemma \ref{AsympCaseIIA} and the asymptotics (\ref{AsympAmCm})-(\ref{AsympBm}), we see that there exists a constant $e_k$ (precisely we have $e_k =  b_k \ (f(0)f(1))^{1/4}$) such that
\begin{equation} \label{Asymp1}
  \frac{C_m-A_m}{B_m} = e_k \ \mu_m^{-\frac{k}{2}} e^{\sqrt{\mu_m}} \ \left(1+ \mathcal  O \left(\frac{1}{\sqrt{\mu_m} }\right)\right) , \mbox{ as }  m \to \infty.
\end{equation}
Hence the term $\left( \frac{C_m-A_m}{B_m} \right)^2$ is exponentially increasing as $m \to \infty$ and dominates $4$ in the square root appearing in  the equation  (\ref{b1}). We thus obtain 
$$
  \frac{\lambda_m^\pm - A_m}{B_m} = \half \left[ \frac{C_m-A_m}{B_m} \mp \left| \frac{C_m-A_m}{B_m} \right| \right] 
   \mp \left| \frac{B_m}{C_m-A_m} \right| + \mathcal O\left( \left(\frac{B_m}{C_m-A_m}\right)^3 \right).
$$
More precisely, if the constant $b_k < 0$ and thus $d_k < 0$, 
\begin{align} 
   \frac{\lambda_m^+ - A_m}{B_m} & =  \frac{C_m-A_m}{B_m} + \mathcal O\left( \frac{B_m}{C_m-A_m} \right), \mbox{ as } m \to \infty\,, \label{Flea+} \\
    \frac{\lambda_m^- - A_m}{B_m}  & = \left| \frac{B_m}{C_m-A_m} \right|  +  \mathcal O\left( \left(\frac{B_m}{C_m-A_m}\right)^3 \right), \mbox{ as } m \to \infty\,. \label{Flea-}
\end{align}
 Similarly, we have
\begin{align}
	& \frac{1}{n_m^+} = \left| \frac{B_m}{C_m-A_m} \right| \ \left(1+  \mathcal O\left( \left(\frac{B_m}{C_m-A_m}\right)^2 \right) \right), \mbox{ as }  m \to \infty\,, \\
	& \frac{1}{n_m^-}  = 1 + \mathcal O\left( \left(\frac{B_m}{C_m-A_m}\right)^2 \right), \mbox{ as } m \to \infty\,.
\end{align}
Then, the result follows easily from (\ref{AsympPsi}) and (\ref{AsympPhi}) and
$$
w_m^\pm(x) = \frac{1}{n_m^\pm} \frac{\Psi(x,-\mu_m,q)}{f^{\frac{1}{4}}(0)} + \frac{\lambda_m^\pm - A_m}{B_m n_m^\pm} \frac{\Phi(x,-\mu_m,q)}{f^{\frac{1}{4}}(1)}\,.
$$
Note that if $b_k > 0$ and thus $d_k >0$, the same equalities hold with $+$ replaced by $-$.
\end{proof}

As a by-product, we deduce 

\begin{coro} \label{LocalizationIIApointwise}
	If  $k \geq 0$ is the smallest integer such that $f^{(k)}(0)  \ne (-1)^k f^{(k)}(1)$,  there exist $C_k>0$ and $m_0>0$ such that for all $m \geq m_0$ and for all $(x,\theta) \in [0,1] \times K$,  either
	\begin{align}
		& |\varphi_m^+(x,\theta)| \leq C_k \left( \mu_m^{\alpha_k -\half} \ e^{-\sqrt{\mu_m}(1+x)} + e^{-\sqrt{\mu_m}(1-x)} \right) |Y_m(\theta)|, \label{LocalizationIIA+} \\
		& |\varphi_m^-(x,\theta)| \leq C_k \left(  e^{-\sqrt{\mu_m}x} + \mu_m^{\alpha_k -\half} \ e^{-\sqrt{\mu_m}(2-x)} \right)  |Y_m(\theta)|, \label{LocalizationIIA-} 
	\end{align}
	or the same estimates hold with $+$ and $-$ inverted. 
\end{coro}

Using H\"{o}rmander's $L^\infty$ bound (\ref{HormanderBound}) together with Corollary \ref{LocalizationIIApointwise}, we prove Theorem \ref{MainIIA}. \\

\vspace{0.3cm}
\textbf{Case II.B}. In this subsection, we are not able to precise the leading term of the Steklov eigeinfunctions $\varphi_m^{\pm}(x, \theta)$ since we only got a lower bound for $|A_m -C_m |$ in Proposition \ref {EstimateBelow}. 

\par\noindent
We prove

\begin{prop}\label{PropIIB}
	Assume that there exists $a \in [0,\half[$ such that $f(x) = f(1-x)$ for all $x \in [0,a]$ and that there exists $0<\delta <\half - a$ such that for all $x \in ]a,a+\delta]$, we have $q_f(x) - q_f(1-x) > 0$. Then, for all $\e>0$, there exist $C_\epsilon$ and $m_\epsilon$ such that for all $m\geq m_\epsilon$ and for all $x \in [0,1]$,
	\begin{align*}
		& |w_m^-(x)| \leq C_\e \left(  e^{-\sqrt{\mu_m}(1 - 2(a+\e)+x)} + e^{-\sqrt{\mu_m}(1-x)} \right), \\
		& |w_m^+(x)| \leq C_\e \left(  e^{-\sqrt{\mu_m}x} +  e^{-\sqrt{\mu_m}(2 - 2(a+\e)-x)} \right), 
	\end{align*}
 The case $q_f(x) - q_f(1-x) < 0$ entails the same results with $+$ and $-$ inverted. 
\end{prop}

\begin{proof}
Assume that for all $x \in ]a,a+\delta]$, we have $q_f(x) - q_f(1-x) > 0$. Using (\ref{AsympAmCm})-(\ref{AsympBm}) and Lemma \ref{EstimateBelow}, we get for all $\e>0$  and $m$ sufficiently large
\begin{equation} \label{a2}
  \left| \frac{C_m - A_m}{B_m} \right| \geq {C_\e} \ e^{(1-2(a+\e)) \sqrt{\mu_m}}. 
\end{equation}
Hence the term $\left( \frac{C_m - A_m}{B_m} \right)^2$ dominates $4$ as $m \to \infty$. Coming back to (\ref{b1}), we infer that 
$$
\frac{\lambda_m^\pm - A_m}{B_m} = \half \left[ \frac{C_m-A_m}{B_m} \mp \left| \frac{C_m-A_m}{B_m} \right| \right] + \mathcal O\left( \frac{B_m}{C_m-A_m} \right).
$$
More precisely, it follows from the proof of Lemma \ref{EstimateBelow} that $\frac{C_m - A_m}{B_m} < 0$ for $m$ large enough. Thus, we have
\begin{align*} 
	\frac{\lambda_m^- - A_m}{B_m} & =  \frac{C_m-A_m}{B_m} + \mathcal O\left( \frac{B_m}{C_m-A_m} \right), \mbox{ as } m \to \infty\,,  \\
	\frac{\lambda_m^+ - A_m}{B_m}  & =  \mathcal O\left( \frac{B_m}{C_m-A_m} \right), \mbox{ as } m \to \infty\,,  
\end{align*}
where the remainder $\mathcal O\left( \frac{B_m}{C_m-A_m} \right) = \mathcal O\left( e^{-(1-2(a+\e)) \sqrt{\mu_m}} \right)$ is exponentially decreasing as $m \to \infty$ due to (\ref{a2}). Similarly, we have
\begin{align*}
	& \frac{1}{n_m^-} = \mathcal O\left( \frac{B_m}{C_m - A_m} \right), \mbox{ as }  m \to \infty, \\
	& \frac{1}{n_m^+}  = \mathcal O \left( 1 \right), \mbox{ as } m \to \infty.
\end{align*}
Mimicking the proof of the case II.A, we obtain easily the result.
\end{proof}

Eventually, using H\"{o}rmander's $L^\infty$ bound (\ref{HormanderBound}), we prove Theorem \ref{MainIIB} as previously. \\

\vspace{0.5cm}

\textbf{Case II.C}. In this case, we prove

\begin{prop}\label{PropIIC}
If  $f(x)$ and $f(1-x)$ have the same Taylor series at $x=0$, there exists a subsequence $(m_k)_{k\geq 0}$ and a constant $C > 0$ such that for all $\e > 0$, there exists $m_\e > 0$ such that for all $m_k \geq m_\e$ and for all $x\in [0,1]$,
	\begin{align*}
		& |w_{m_k}^-(x)| \leq C \left( \e e^{-\sqrt{\mu_{m_k}}x} + e^{-\sqrt{\mu_{m_k}}(1-x)} \right) , \\
		& |w_{m_k}^+(x)| \leq C \left(  e^{-\sqrt{\mu_{m_k}}x} + \e e^{-\sqrt{\mu_{m_k}}(1 - x}  \right) .
	\end{align*}	 

\end{prop}	

\begin{proof}
Recall from Proposition \ref{SubsequenceIIC} that we can find a subsequence $(m_k)_{k \geq 0}$ such that (for instance), 
$$
\frac{A_{m_k} - C_{m_k}}{B_{m_k}} \to - \infty,
$$
but we are unable to give more precision of the decay rate. Thus, micmicking the proof of  the cases II.A and  II.B, we obtain
\begin{align*} 
	\frac{\lambda_{m_k}^- - A_{m_k}}{B_{m_k}} & =  \frac{C_{m_k}-A_{m_k}}{B_{m_k}} + \mathcal O\left( \frac{B_{m_k}}{C_{m_k}-A_{m_k}} \right), \quad \mbox{ as } k \to \infty\,,  \\
	\frac{\lambda_{m_k}^+ - A_{m_k}}{B_{m_k}}  & =  \mathcal O\left( \frac{B_{m_k}}{C_{m_k}-A_{m_k}} \right), \quad \mbox{ as } k  \to \infty\,. 
\end{align*}
Thus, we have
\begin{align*}
	& \frac{1}{n_{m_k}^-} = \mathcal O\left( \frac{B_{m_k}}{C_{m_k} - A_{m_k}} \right), \quad k \to \infty, \\
	& \frac{1}{n_{m_k}^+}  = \mathcal O\left( 1 \right), \quad k \to \infty ,
\end{align*}

and we finish the proof  as previously.

\end{proof}

Using H\"{o}rmander's $L^\infty$ bound (\ref{HormanderBound}), we prove Theorem \ref{MainIIC}. \\

\subsection{On the choice of normalization}
 At last, let us discuss the normalization question. Recall that the Steklov eigenfunctions $\varphi_m^\pm$ were defined such that their trace $\phi_m^\pm = (\varphi_m^\pm)_{|\partial M}$ on the boundary $\partial M$ are normalized by  
$$
  \| \phi_m^\pm \|_{L^2(\partial M)} = 1. 
$$ 
We consider now the Steklov eigenfunctions $\tilde{\varphi}_m^\pm$ whose $L^2$ norm is normalized on the whole manifold $M$, \textit{i.e.}  
$$
  \| \tilde{\varphi}_m^\pm \|_{L^2(M)} = 1.  
$$
The Steklov eigenfunctions $\varphi_m^\pm$ and $\tilde{\varphi}_m^\pm$ are obviously connected by  
$$
  \tilde{\varphi}_m^\pm = \frac{\varphi_m^\pm}{\| \varphi_m^\pm \|_{L^2(M)}}. 
$$
We can deduce estimates on $\tilde{\varphi}_m^\pm$ from the corresponding estimates on $\varphi_m^\pm$ using the following result.

\begin{thm}\label{normthm}
 There exist constants $0 < c_1 < c_2$  and $m_0>0$ such that for $m\geq m_0$ 
\begin{equation} \label{Norm}
\frac{c_1}{\mu_m^{\frac{1}{4}}} \leq \| \varphi_m^\pm \|_2 \leq \frac{c_2}{\mu_m^{\frac{1}{4}}}\,.
\end{equation}
\end{thm}

\begin{proof}
We only prove the assertion in the case I corresponding to symmetric warped product since the other cases use similar arguments. We shall use the expression (\ref{vm}), (\ref{SteklovEigen}) and (\ref{NormalizedSteklovEigenfunction}) for the normalized eigenfunctions $\tilde{\varphi}_m^\pm$. 

Under the assumption of symmetry on $f$, we recall that $n_m^\pm = \sqrt{2}$ and that
$$
  w_m^\pm(x) = \frac{\Psi(x,-\sqrt{\mu_m},q) \mp \Phi(x,-\sqrt{\mu_m},q)}{\sqrt{2} f^{\frac{1}{4}}(0)}. 
$$
Using (\ref{Psi})-(\ref{Phi}), an easy calculation gives
\begin{equation} \label{Psi-Phi}
w_m^\pm(x) = \frac{e^{-\sqrt{\mu_m}x} \left( 1 + \mathcal O(\frac{1}{\sqrt{\mu_m}}) \right) \mp e^{-\sqrt{\mu_m}(1-x)} \left( 1 + \mathcal O(\frac{1}{\sqrt{\mu_m}}) \right) }{\sqrt{2} f^{\frac{1}{4}}(0)},
\end{equation}
as $m \to \infty$ uniformly in $x \in [0,1]$. 
%

Now, using (\ref{Psi-Phi}), we  can complete the proof of \eqref{Norm}. 
Indeed, we first note that 
$$
m_1\int_0^1 |w_m^\pm(x)|^2 dx \leq  \| \varphi_m^\pm \|_2^2 = \int_0^1 |w_m^\pm(x)|^2 f(x) dx \leq m_2 \int_0^1 |w_m^\pm(x)|^2 dx, 
$$
for some constants $0 < m_1 < m_2$. Then using (\ref{Psi-Phi}) and the notation $I_m(x) = \left( 1 + \mathcal O(\frac{1}{\sqrt{\mu_m}}) \right)$, we have
\begin{align*}
\int_0^1 |w_m^\pm(x)|^2 dx & = \frac{1}{2 f^\half(0)} \int_0^1 \left| 	e^{-\sqrt{\mu_m}x} I_m(x) \mp e^{-\sqrt{\mu_m}(1-x)} \tilde{I}_m(x) \right|^2 dx, \\ 
& =  \frac{1}{2 f^\half(0)}  \Big[ \int_0^1 e^{-2\sqrt{\mu_m}x} I_m^2(x) dx  \mp 2 e^{-\sqrt{\mu_m}} \int_0^1  I_m(x) \tilde{I}_m(x) dx \\
& \hspace{2cm} + \int_0^1 e^{-2\sqrt{\mu_m}(1-x)} \tilde{I}_m^2(x)  dx \Big]\,.
\end{align*}	
The first and third integrals are estimated for all large enough $m$ by
$$
\frac{C_1}{\sqrt{\mu_m}}  \leq \int_0^1 e^{-2\sqrt{\mu_m}x} I_m^2(x) dx, \  \int_0^1 e^{-2\sqrt{\mu_m}(1-x)} \tilde{I}_m^2(x)  dx  \leq \frac{C_2}{\sqrt{\mu_m}}\,, 
$$
for some constants $0 < C_1 < C_2$, whereas the second term is bounded by  $e^{-\sqrt{\mu_m}}$. Putting all these estimates together, we get (\ref{Norm}). 

\end{proof}


\appendix

\section{Some results on the theory of Weyl-Titchmarsh functions} \label{WeylTitchmarsh}

We consider the class of regular Schr\"odinger equations on the interval $[0,1]$ given by
\begin{equation} \label{Equation}
	-v'' + q(x) v = z v,  
\end{equation}
where $q \in L^1([0,1])$ is a real potential  and $z\in \mathbb C$. 

For all $z \in \C$, we define the two Fundamental Systems of Solutions (FSS)
$$
\{ c_0(x,z,q), s_0(x,z,q)\}, \quad \{ c_1(x,z,q), s_1(x,z,q)\},
$$
of (\ref{Equation}) by imposing the Cauchy conditions
\begin{equation} \label{FSS}
	\left\{ \begin{array}{cccc} c_0(0,z,q) = 1, & c_0'(0,z,q) = 0, & s_0(0,z,q) = 0, & s_0'(0,z,q) = 1, \\
		c_1(1,z,q) = 1, & c'_1(1,z,q) = 0, & s_1(1,z,q) = 0, & s'_1(1,z,q) = 1. \end{array} \right.
\end{equation}
It follows from (\ref{FSS}) that
\begin{equation} \label{Wronskian-FSS}
	W(c_0, s_0) = 1, \quad W(c_1, s_1) = 1, \quad \forall z \in \C,
\end{equation}
where $W(u,v) = uv' - u'v$ is the Wronskian of $u,v$.\\
 Moreover, the FSS $\{ c_0(x,z,q), s_0(x,z,q)\}$ and $\{ c_1(x,z,q), s_1(x,z,q)\}$ are entire functions of order $\half$ with respect to the variable $z \in \C$.

We then define the characteristic function of (\ref{Equation}) with Dirichlet boundary conditions by
\begin{equation} \label{Char}
	\Delta_q(z) = W(s_0, s_1) = s_0(1,z,q) = -s_1(0,z,q).
\end{equation}
The characteristic function $z \mapsto \Delta_q(z)$ is also an entire function of order $\half$ with respect to $z$ and its zeros $(\alpha_{k})_{k \geq 1}$ correspond to Dirichlet eigenvalues of the selfadjoint operator $-\frac{d^2}{dx^2} + q$. The eigenvalues $\alpha_k$ are thus real, simple and are ordered by  $\alpha_1 < \alpha_2 < \dots $.

We next define two Weyl-Titchmarsh functions by the following classical prescriptions. Let the Weyl solutions $\Psi$ and $\Phi$ be the unique solutions of (\ref{Equation}) having the form
\begin{equation} \label{WeylFunction}
	\begin{array}{c} \Psi(x,z,q) = c_0(x,z,q) + M_q(z) s_0(x,z,q), \\
		\Phi(x,z,q) = c_1(x,z,q) - N_q(z) s_1(x,z,q), \end{array}
\end{equation}
which satisfy the Dirichlet boundary condition at $x = 1$ and $x=0$ respectively. Then a short calculation using (\ref{FSS}) shows that the Weyl-Titchmarsh functions $M_q(z)$ and $N_q(z)$ are uniquely defined by
\begin{equation} \label{WT}
	M_q(z) = - \frac{W(c_0, s_1)}{\Delta_q(z)}, \quad N_q(z) = -\frac{W(c_1, s_0)}{\Delta_q(z)}.
\end{equation}
We introduce the functions 
\begin{equation} \label{DE}
  D_q(z) = W(c_0, s_1) = c_0(1,z,q) = s'_1(0,z,q), \ E_q(z) = W(c_1, s_0) = c_1(0,z,q) = s'_0(1,z,q),
\end{equation}
which also turn out to be entire functions of order $\half$ in $z$. We then have 
\begin{equation}\label{MN}
  M_q(z) = - \frac{D_q(z)}{\Delta_q(z)}, \quad N_q(z) = - \frac{E_q(z)}{\Delta(z)}.
\end{equation}
Observe also that
$$
  \Psi(x,z,q) = -\frac{s_1(x,z,q)}{\Delta_q(z)}, \quad \Phi(x,z,q) = \frac{s_0(x,z,q)}{\Delta_q(z)}.
$$

\begin{rem}[Symmetry with respect to $\half$] \label{SymmetryWT}
Given a potential $q \in L^1(0,1)$, define the symmetrized potential
$$
\check{q}(x) = q(1-x). 
$$
Then we can check easily that
\begin{align*}
  & c_0(x,z,\check{q}) = \check{c_1}(x,z,q), \quad c_1(x,z,\check{q}) = \check{c_0}(x,z,q), \\
  & s_0(x,z,\check{q}) = -\check{s_1}(x,z,q), \quad  s_1(x,z,\check{q}) = -\check{s_0}(x,z,q).
\end{align*}
This implies in turn that
\begin{equation} \label{Sym}
  \Delta_{\check{q}} = \Delta_q, \quad M_{\check{q}} = N_q.
\end{equation}
In particular, the $N$ function corresponding to a potential $q$ plays the role of the $M$ function corresponding to the symmetrized potential $\check{q}$. This emphasizes the natural symmetry about $\half$ of the problem. 
\end{rem}

We now collect some results involving the functions $\Delta_q(z)$, $D_q(z)$, $E_q(z)$, $M_q(z)$ and $N_Q(z)$ in the form we shall need later.

\begin{prop} \label{AsympFSS}
	The FSS $\{ c_0(x,z,q), s_0(x,z,q)\}$ and $\{ c_1(x,z,q), s_1(x,z,q)\}$ have the following asymptotics uniformly with respect to $x \in [0,1]$ as the variable $\rho = \sqrt{-z} \to \infty$ in the complex plane $\C$.
	\begin{equation} \label{Asymp-0}
		\left\{ \begin{array}{ccc}
			c_0(x,z,q) & = & \cosh(\sqrt{-z} x) +  \mathcal O\left( \frac{e^{|\Re\sqrt{-z}| x}}{\sqrt{-z}} \right), \\
			c'_0(x,z,q) & = & \sqrt{-z} \sinh(\sqrt{-z} x) +  \mathcal O\left( e^{|\Re(\sqrt{-z})| x} \right), \\
			s_0(x,z,q) & = & \frac{\sinh(\sqrt{-z} x)}{\sqrt{-z}} +  \mathcal O\left( \frac{e^{|\Re(\sqrt{-z})| x}}{z} \right), \\
			s'_0(x,z,q) & = & \cosh(\sqrt{-z} x) +  \mathcal O\left( \frac{e^{|\Re(\sqrt{-z})| x}}{\sqrt{-z}} \right),
		\end{array} \right.
	\end{equation}
	and
	\begin{equation} \label{Asymp-1}
		\left\{ \begin{array}{ccc}
			c_1(x,z,q) & = & \cosh(\sqrt{-z} (1-x)) +  \mathcal O\left( \frac{e^{|\Re(\sqrt{-z})| (1-x)}}{\sqrt{-z}} \right), \\
			c'_1(x,z,q) & = & -\sqrt{-z} \sinh(\sqrt{-z} (1-x)) +  \mathcal O\left( e^{|\Re(\sqrt{-z})| (1-x)} \right), \\
			s_1(x,z,q) & = & -\frac{\sinh(\sqrt{-z} (1-x))}{\sqrt{-z}} +  \mathcal O\left( \frac{e^{|\Re(\sqrt{-z})| (1-x)}}{z} \right), \\
			s'_1(x,z,q) & = & \cosh(\sqrt{-z} (1-x)) +  \mathcal O\left( \frac{e^{|\Re(\sqrt{-z})| (1-x)}}{\sqrt{-z}} \right).
		\end{array} \right.
	\end{equation}
\end{prop}
\begin{proof}
	These asymptotics are classical and can be found in \cite{PT1987} (Theorem 3, p. 13). 
\end{proof}

\begin{coro} \label{AsympDE}
	1. For each fixed $x \in [0,1]$, the fundamental systems of solutions $\{ c_0(x,z,q), s_0(x,z,q)\}$ and $\{ c_1(x,z,q), s_1(x,z,q)\}$ are entire functions of order $\frac{1}{2}$ with respect to the variable $z$. \\
	2. The characteristic function $\Delta_q(z)$ and the functions $D_q(z)$ and $E_q(z)$ are entire functions of order $\frac{1}{2}$ with respect to the variable $z$. \\
	3. We have the following asymptotics in the complex plane $\C$:
	$$
	\left\{ \begin{array}{c} \Delta_q(z) = \frac{\sinh{\sqrt{-z}}}{\sqrt{-z}} +  \mathcal O\left( \frac{e^{|\Re(\sqrt{-z})|}}{z} \right), \\
	D_q(z) = \cosh (\sqrt{-z}) +  \mathcal O\left(  \frac{e^{|\Re(\sqrt{-z})|}}{\sqrt{-z}} \right), \\ E_q(z) = \cosh (\sqrt{-z}) +  \mathcal O\left(  \frac{e^{|\Re(\sqrt{-z})|}}{\sqrt{-z}} \right).
	\end{array} \right.
	$$
\end{coro}
\begin{proof}
	The proof of 1., 2. and 3. follows directly from (\ref{Char}), (\ref{WT}) and Lemma \ref{AsympWT}.
\end{proof}

\begin{coro} \label{Hadamard}
	The characteristic function $\Delta(z)$ and the functions $D(z)$ and $E(z)$ can be written as
	\begin{equation} \label{HadFacto}
		\begin{array}{ccc}
			\Delta_q(z) & = & \Delta_q(0) \ds\prod_{k = 1}^\infty \left( 1 - \frac{z}{\alpha_{k}} \right), \\
			D_q(z) & = & D_q(0) \ds\prod_{k = 1}^\infty \left( 1 - \frac{z}{\beta_{k}} \right), \\
			E_q(z) & = & E_q(0) \ds\prod_{k = 1}^\infty \left( 1 - \frac{z}{\gamma_{k}} \right),
		\end{array}
	\end{equation}
	where $(\alpha_{k})_{k \geq 1}$, $(\beta_{k})_{k \geq 1}$ and $(\gamma_{k})_{k \geq 1}$ are the zeros of the entire functions $\Delta_q(z)$, $D_q(z)$ and $E_q(z)$ respectively.
\end{coro}
\begin{proof}
	This is a direct consequence of Hadamard's factorization Theorem (see \cite{Boa1954, Lev1996}) for the entire functions $\Delta_q(z)$, $D_q(z)$ and $E_Q(z)$ of order $\frac{1}{2}$.
\end{proof}

We now recall some important facts about the Weyl-Titchmarsh function $M  = M_q$ obtained by B.~Simon in \cite{GS2000a, Si1999}. 

\begin{thm} \label{A-function}
There exists a function $A \in L^1_{loc}(\R^+)$ such that for $k \in \N^*$ and $a < 2k$
\begin{equation} \label{ARepresentation}
	M(-\kappa^2) = -\kappa - \int_0^a A(\alpha) e^{-2\kappa\alpha} d\alpha - 2 \sum_{j=1}^k (\kappa + j \int_0^1q(x) dx) e^{-2 j \kappa} + \tilde{\mathcal O} (e^{-2a\kappa}) \mbox{ as } \kappa  \to +\infty. 
\end{equation}		
Moreover, $A-q$ is continuous on $[0,1]$ and obeys for all $\alpha \in [0,1]$
\begin{equation} \label{EstA1}
	|A(\alpha) - q(\alpha)| \leq Q(\alpha)^2 e^{\alpha Q(\alpha)}, \quad Q(\alpha) = \int_0^\alpha |q(s)|ds.
\end{equation}
We also have
\begin{equation} \label{EstA2}
	|A(\alpha,q) - A(\alpha,\tilde{q})| \leq \|q - \tilde{q}\|_{L^1} [Q(\alpha) + \tilde{Q}(\alpha)] e^{\alpha [Q(\alpha)+\tilde{Q}(\alpha)]}.
\end{equation}
At last, B.~Simon showed that the potential $q$ on $[0,a]$ is a function of $A$ on $[0,a]$. More precisely, if $q$ and $\tilde{q}$ are two potentials, let $A$ and $\tilde{A}$ be their $A$-functions. Then
$$
  A(\alpha) = \tilde{A}(\alpha), \ \forall \alpha \in [0,a] \ \Longleftrightarrow \ q(x) = \tilde{q}(x), \ \forall x \in [0,a].
$$
\end{thm}

As a consequence, B.~Simon proved the following theorem.

\begin{thm}[local Borg-Marchenko]\label{BMlocal}
Let $q$ and $\tilde{q}$ two potentials in $L^1(0,1)$. Then $q(x) = \tilde{q}(x)$ for all $x \in [0,a]$ if and only if $M(-\kappa^2) = \tilde{M}(-\kappa^2) + \tilde{\mathcal O} (e^{-2a\kappa})$ as $\kappa \to +\infty$. 
\end{thm}

\begin{thm}[Asymptotics of $M$] \label{AsympWT}
If $q \in L^1(0,1)$, then
$$
  M(-\kappa^2) = -\kappa - \int_0^1 q(x) e^{-2x} dx  + o\left( \frac{1}{\kappa} \right) \mbox{ as } \kappa  \to +\infty.
$$	
Assume moreover, that $q \in C^k([0,\delta)$. Then
\begin{equation} \label{AsympM}
	M(-\kappa^2) = -\kappa - \sum_{j=0}^k \beta_j \kappa^{-j-1} +\mathcal  O(\kappa^{-k-2}) \mbox{ as } \kappa  \to +\infty.
\end{equation}	 
Here the constants $\beta_j$ can be computed inductively by $\beta_j = \beta_j(0)$ where for all $j \geq 0$
$$
   \beta_0(x) = \half q(x), \quad \beta_{j+1}(x) = \half \beta_j'(x) + \half \sum_{\ell=0}^j \beta_\ell (x) \beta_{j-\ell}(x). 
$$
\end{thm}
Note that the results concerning the $M$ function can be translated to the $N$ function easily thanks to Remark \ref{SymmetryWT}. In particular, for smooth potential $q$, the asymptotics of $N$ are given by
\begin{equation} \label{AsympN}
	N(-\kappa^2) = -\kappa - \sum_{j=0}^k \gamma_j \kappa^{-j-1} + \mathcal O(\kappa^{-k-2}) \mbox{ as } \kappa  \to +\infty.
\end{equation}	 
Here the constants $\gamma_j$ can be computed inductively by $\gamma_j = \gamma_j(0)$ where for all $j \geq 0$
$$
\gamma_0(x) = \half \check{q}(x), \quad \gamma_{j+1}(x) = \half \gamma_j'(x) + \half \sum_{\ell=0}^j \gamma_\ell (x) \gamma_{j-\ell}(x). 
$$

Eventually, we shall need the following formula. 

\begin{prop} \label{lemmaLinkMN}
$$	
  M_q(z) - N_q(z) = \int_0^1 \left( q(x) - q(1-x) \right) \Psi(x,z,q) \Phi(1-x,z,q) dx. 
$$
\end{prop}

\begin{proof}
We compute
$$
  \frac{d}{dx} W(s_1(x,z,q), s_0(1-x,z,q)) = -  \left( q(x) - q(1-x) \right) s_1(x,z,q) s_0(1-x,z,q). 
$$	  
Integrating between $0$ and $1$, we get
$$
  s_1(0,z,q) s'_0(1,z,q) + s'_1(0,z,q) s_0(1,z,q) = -  \int_0^1 \left( q(x) - q(1-x) \right) s_1(x,z,q) s_0(1-x,z,q) dx. 
$$
Using (\ref{Char}) and (\ref{DE}), we obtain
$$
  \Delta_q(z) (D_q(z) - E_q(z)) = - \int_0^1 \left( q(x) - q(1-x) \right) s_1(x,z,q) s_0(1-x,z,q) dx,
$$ 
which entails the formula.  
\end{proof}

\section{A comparison with the flea on the elephant phenomenon for Schr\"odinger  operators with a double well potential. } \label{FleaElephant}
The aim of this appendix is to present shortly (with some cheating in the presentation to avoid technicalities) the flea on the elephant effect as initially considered by  G. Jona-Lasinio, F. Martinelli and E. Scoppola \cite{JLMS1981} and then developed more systematically in \cite{HeSj1985} and with another approach  by
B. Simon in \cite{Si1985}  who gives to this effect its evocative name.

We start from the double well symmetric situation with a potential $V_0$, satisfying $$ \inf V_0=0 \mbox{  and } \lim_{|x| \rightarrow +\infty} V_0 (x)=+\infty\,,$$  and  having two  symmetric non degenerate 
minima denoted by $ U_{\ell}$ and $ U_{r}$.   We are interested in the description,  in the limit $h\rightarrow 0_+$,  of the two first eigenvalues of the selfadjoint realization in $\mathbb R^d$ of 
$$ P_h:=-h^2 \Delta + V_0\,.$$
 For the symmetry,  we assume for example that $$V_0(x_1,\dots,x_{d-1}, x_d)= V_0(x_1,\dots,x_{d-1}, -x_d)$$ and that $U_\ell $ and $U_r$ are exchanged by this symmetry.   Under these assumptions we know that the operator has compact resolvent, that its spectrum consists of a sequence $\lambda_j(h)$ ($j\geq 1$) tending to $+\infty$  and that the ground state energy $\lambda_1(h)$ (i.e. the first eigenvalue or principal eigenvalue) is simple. Moreover the symmetry  implies  that the first eigenfunction is symmetric. Using an harmonic approximation near $U_\ell$ and $U_r$ we can first show that   there are two eigenvalues $\lambda_1 (h)$ and $\lambda_2 (h)$
 which are close modulo $\mathcal O (h^\frac 32)$ (in the limit $h\rightarrow 0$) to the first eigenvalue $\lambda_0(h)$ of the Harmonic oscillator 
 $$H_0 (y, hD_y):=-h^2 \Delta_y +\frac 12 \langle  {\rm Hess} V_0 (U_\ell) y,y\rangle \,.$$
{This can be explicitly} computed as 
 $$
 \lambda_0(h) = h \gamma_0 \mbox{ with  } \gamma_0 = {\rm Tr} [ ( \frac 12 {\rm Hess} V_0)^\frac 12] (U_\ell)\,.
 $$
 Moreover one can show that there exists $\epsilon_0 >0$ and $h_0>0$ such that $\lambda_j (h) > h (\gamma_0 +\epsilon_0)$ for $j\geq 3$ and $h\in (0,h_0]$.
 The aim is then to analyze the splitting $\lambda_2(h)-\lambda_1(h)$ in the limit $h\rightarrow 0$.\\ The idea is to recover   these two eigenvalues as eigenvalues  of a $2\times 2$ matrix, which is called the interaction matrix,  corresponding to the expression of the Hamiltonian reduced to the 
spectral space associated with $\sigma ( P_h) \cap [0, (\gamma_0+\epsilon_0) h)$ in an adapted basis of this two dimensional vector space $E(h)$.  This  basis is constructed such that $u_{\ell}$ is strongly localized in $U_\ell$ and $u_{r}$ is defined as the symmetric of $u_\ell$. This localization is defined in the following sense. There is a natural notion of distance (the Agmon distance) $d_{V_0}$  associated with  the degenerate metric $V_0(x) g_0$ where $g_0$ is the Euclidean metric.  Then $u_{\ell}$ has the following property for some constant $C>0$ 
$$
| u_\ell (x,h)| \leq C h^{-C} \exp - d_{V_0} (x, U_\ell)/h\,,
$$ 
and can in addition be well approximated by a so called WKB approximation in a neighborhood of the set ${\rm Geo}(U_\ell,U_r)$ of the minimal geodesics (for the Agmon distance) between $U_\ell$ and $U_{r}$ outside a small ball around $U_r$. \\ 
This matrix has the form
$$
M =\left(\begin{array}{cc} a &b\\b&a\end{array}
\right)
$$
where $a$ is asymptotically equivalent to the first level of the Harmonic approximation and $b$ is the so called tunneling interaction:
\begin{equation}\label{eq:b}
b = h^{-\nu} \beta(h) \exp - S /h\,,
\end{equation}
where $S$ is the Agmon distance between $U_\ell $ and $U_r$, $\nu >0$ and $\beta (h) \sim \beta_0 \neq 0$.\\
Clearly, the  eigenvalues are $a\pm b$ corresponding to the eigenvectors $\frac{1}{\sqrt{2}}(1,1)$ and  $\frac{1}{\sqrt{2}}(1,-1)$. From this we deduce  that the two first eigenfunctions are equilocalized in the neighborhood of the minima and 
  $\frac{1}{\sqrt{2}} (u_1+ u_2)$ is strongly localized near $U_\ell$ and $\frac{1}{\sqrt{2}} (u_1-u_2)$ is strongly localized in $U_r$.\\
We can now explain the flea of the elephant effect.
We  consider as new potential $$  V_\delta (x) = V_0(x) + \delta \, w(x)$$
 where 
$ 0 \leq w \in C_0^\infty(\mathbb R^d)$, and $\delta \geq 0$. \\
We assume that  
\begin{equation}
0 < d_{V_0} (U_\ell, {\rm supp\, }w ) < \frac 12 d_{V_0} (U_\ell,U_r) < d_{V_0} (U_r,{\rm supp\,}  w) \,.
\end{equation}
and   that 
\begin{equation} \label{eq:geod} {\rm supp\,} w \cap  {\rm Geod}(U_\ell,U_r) =\emptyset\,.
\end{equation} In this situation, one can again reduce\footnote{$u_\ell$ and $u_r$ are slightly modified in a very well controlled way} the analysis of the spectrum in $[0, (\gamma_0 +\epsilon_0)h)$  to the analysis of the two eigenvalues of a  symmetric $ 2\times 2$ matrix 
$ M^\delta= (M^\delta_{ij}) $ where the principal term in the off-diagonal coefficient $ M_{12} ^\delta$ is essentially unchanged
\begin{equation}\label{eq:offdiag}
M_{12}^\delta = b  + \mathcal O (e^{-S -\eta /h}) \mbox{ for some } \eta >0\,.
\end{equation}
Similarly 
\begin{equation}\label{eq:diag2}
M_{22}^\delta = a  + \mathcal O (e^{-S -\eta /h}) \mbox{ for some } \eta >0\,.
\end{equation}
But $ M_{11}^\delta $ can be estimated from below by (for $ 0 < \delta < \delta_0 $)
\begin{equation} \label{eq:diag1} 
M_{11}^\delta  \geq  a  + \delta  \exp - \frac{ 2  d_{V_0} (U_\ell, {\rm supp\, }w ) -\epsilon}{h}\,,\, \forall \epsilon >0\,.
\end{equation}
Formally, the main term of the perturbation is indeed 
$$  
\delta \int w(x)  u_\ell(x) ^2 \, dx \,,
$$
and we use  a lower bound for the decay of $u_\ell $ on the support of $w$.\\
This implies the existence of $ \delta_1 >0$ such that, for $0 <\delta <\delta_1$, we have as $ h \rightarrow 0$, 
\begin{equation}\label{eq:dom}
(M_{11}^\delta - M_{22}^\delta)^2 \gg (M_{12}^\delta )^2\,.
\end{equation}
 Actually this is still true with $\delta = \exp - \alpha/h $ for $\alpha >0$ small enough. \\

In this situation the eigenvalues of $ M^\delta$ satisfy for some arbitrarily small $\eta >0$
\begin{equation} 
\lambda_-^\delta = a + \mathcal O_\eta (e^{-S +\eta /h})\,,
\end{equation}
and 
\begin{equation}
\lambda_+^\delta \geq  a  +  \delta \exp -   \frac{2 d_{V_0} (U_\ell, {\rm supp\, }w ) }{ h}\ 
\end{equation}
and the corresponding eigenvectors are  up to exponentially small terms are $ (0,1)$ and  $ (1,0)$. \\
This shows a  localization of the ground state near $ U_r$. The ground state is actually exponentially small near $ U_\ell$. The second eigenfunction will be localized near $ U_\ell$.\\
\begin{rem}
\begin{itemize}~
\item When \eqref{eq:geod} is not satisfied we loose \eqref{eq:offdiag} but keep instead an upper bound similar to the estimate we have on $b$ (see \eqref{eq:b}) in the weaker  form
$$
|M_{12}^\delta| \leq C_\epsilon \exp - (S-\epsilon)/h\,,\, \forall \epsilon >0\,.
$$
Actually because $\delta \geq 0$ and $w\geq 0$, we could replace $S$ by $S_\delta > S$ but we do not need this improvement.
\eqref{eq:diag2} and \eqref{eq:diag1} are unchanged and this permits to show that
  \eqref{eq:dom} holds also in this case. In particular in dimension 1, we can consider a perturbation $w$ with support between the two minima. This was actually the case initially considered in \cite{JLMS1981}.
\item
To understand the analogy with the Steklov problem, we have to consider  that the two minima in the double well problem are the two components of the boundary. The perturbation of $V_0$ corresponds to the perturbation of $f$.
\item More recently a numerical analysis of the effect is proposed in \cite{BGZ2019}.
 
\end{itemize}
\end{rem}

\bibliographystyle{acm}
\bibliography{Biblio}

\vspace{0.5cm}

\noindent \footnotesize{DEPARTEMENT DE MATHEMATIQUES. UMR CNRS 8088. UNIVERSITE DE CERGY-PONTOISE. 95302 CERGY-PONTOISE. FRANCE. \\
	\emph{Email adress}: thierry.daude@u-cergy.fr \\

\noindent LABORATOIRE DE MATHEMATIQUES JEAN LERAY. UMR CNRS 6629. 2 RUE DE LA HOUSSINIERE BP 92208. F-44322 NANTES CEDEX 03 \\
\emph{Email adress}: Bernard.Helffer@univ-nantes.fr \\

\noindent LABORATOIRE DE MATHEMATIQUES JEAN LERAY. UMR CNRS 6629. 2 RUE DE LA HOUSSINIERE BP 92208. F-44322 NANTES CEDEX 03 \\
\emph{Email adress}: francois.nicoleau@univ-nantes.fr}

\end{document}